\documentclass[11pt]{article}
\usepackage{amsfonts}
\usepackage{graphics}
\usepackage{indentfirst}
\usepackage{color}
\usepackage{cite}
\usepackage{latexsym}
\usepackage[paper=a4paper, left=1.6cm, right=1.6cm, top=1.8cm, bottom=1.6cm, headheight=5.5pt, footskip=0.8cm, footnotesep=0.8cm, centering, includefoot]{geometry}
\usepackage{amsmath}
\allowdisplaybreaks
\usepackage{amssymb}
\usepackage[colorlinks, linkcolor=red]{hyperref}
\hypersetup{urlcolor=red, citecolor=red}
\usepackage[dvips]{epsfig}
\usepackage{amscd}

\usepackage{amsthm}
\usepackage{mathrsfs}
\usepackage{verbatim}
\newtheorem{Theorem}{Theorem}[section]
\newtheorem{Lemma}{Lemma}[section]

\newtheorem{Remark}{Remark}[section]
\newtheorem{Proposition}{Proposition}[section]

\DeclareMathOperator{\loc}{loc}
\DeclareMathOperator{\divv}{div}

\newcommand{\na}{\nabla}

\newcounter{RomanNumber}

\def\be{\begin{equation}}
\def\en{\end{equation}}
\def\bs{\begin{split}}
\def\es{\end{split}}

\makeatletter
\@addtoreset{equation}{section}
\makeatother
\makeatletter
\@addtoreset{equation}{section}
\makeatother

\title{Global well-posedness for the compressible Navier--Stokes equations with vacuum and smallness on coefficient-coupled scaling invariant quantity
\thanks{This research was partially supported by  National Natural Science Foundation of China (Nos. 12301287, 12371227).}
}

\author{Hao Xu$\,^{\rm 1}\,$,\
Xin Zhong$\,^{\rm 2}\,$ {\thanks{Corresponding author. E-mail addresses:
haoxu@lzu.edu.cn (H. Xu), xzhong1014@amss.ac.cn (X. Zhong).}}\date{}\\
\footnotesize $^{\rm 1}\,$
School of Mathematics and Statistics, Lanzhou University,
Lanzhou 730000, P. R. China\\
\footnotesize $^{\rm 2}\,$ School of Mathematics and Statistics, Southwest University, Chongqing 400715, P. R. China}

\begin{document}
\maketitle

\begin{abstract}
We investigate the Cauchy problem of three-dimensional compressible
Navier--Stokes equations with far-field vacuum. Based on delicate energy estimates and structures of the systems under consideration, we show the global well-posedness and decay rates of strong solutions provided that some coefficient-coupled scaling invariant initial quantity is suitably small. In particular, {\it our smallness conditions are independent of any initial data and known parameters in the systems}. Moreover, there is no need to require the compatibility conditions on the initial data. Our results improve previous works.
\end{abstract}



\tableofcontents

\section{Introduction}\label{sec1}

The non-isentropic compressible Navier--Stokes equations in $\mathbb{R}^3$ are written as follows:
\begin{equation}\label{1.1}
\begin{cases}
\rho_t +\divv(\rho u) = 0,\\
(\rho u)_t +\divv(\rho u \otimes u) =\divv\mathbb{T},\\
( \rho E)_t +\divv(\rho E u ) =\divv(\mathbb{T} u)+\kappa\Delta\theta,
\end{cases}
\end{equation}
where the unknown functions $\rho$, $u=(u^1,u^2,u^3)^\top$, and $\theta$ are the fluid density, velocity, and temperature, respectively.
The total energy $E$ is given by
\begin{equation*}
E \triangleq e + \frac{1}{2} |u|^2,  \ \  e = \frac{{\rm R} }{\gamma-1} \theta \triangleq {\rm c_V} \theta ,
\end{equation*}
where $e$ is the specific internal energy and the positive parameters $\gamma>1, {\rm R}, \kappa$, and ${\rm c_V}$ are the adiabatic exponent, perfect gas constant, heat-conductivity coefficient, and specific heat at constant volume, respectively.
The Cauchy stress $\mathbb{T}$ is determined through the Newton rheological law
\begin{equation*}
\mathbb{T} \triangleq \mu \big(\nabla u + (\nabla u)^\top \big) + \lambda\divv u \mathbb{I}_3 - P(\rho, e) \mathbb{I}_3,
\end{equation*}
where $\mathbb{I}_3$ is a $3 \times 3$ unit matrix, $\mu$ and $\lambda$ are the viscosity coefficients  satisfying
\begin{equation*}
\mu>0, \ \  3 \lambda + 2 \mu \geq 0.
\end{equation*}
The state equation of the pressure $P = P(\rho, e)$ is of ideal polytropic gas type
\begin{equation*}
P(\rho,e) = (\gamma-1) \rho e \triangleq{\rm R} \rho \theta.
\end{equation*}
For isentropic fluids, the compressible Navier--Stokes equations become
\begin{equation}\label{1.2}
\begin{cases}
\rho_t + \divv(\rho u) = 0,\\
(\rho u)_t +  \divv(\rho u \otimes u)+\nabla P= \mu \Delta u + (\mu + \lambda) \nabla \divv u.
\end{cases}
\end{equation}
Here $P$ satisfies the equation of state of an ideal fluid
\begin{equation*}
P = a \rho^{\gamma}, \ \ a > 0, \ \ \gamma > 1.
\end{equation*}

The compressible Navier--Stokes system is one of the basic physical models arising in fluid mechanics. The mathematical theory of such a model is well developed in the last decades due to its physical importance and mathematical challenge.
There are so many known results concerning the global existence of solutions to
\eqref{1.1} and \eqref{1.2}. In the absence of vacuum (i.e., $\rho=0$),
please refer to \cite{CD10,CCZ10,Da00,DM17,MN1980,Hoff95,Hoff95*,Ho1997} and references contained therein.

However, as pointed out by many mathematicians (see, e.g., \cite{DM23,HLX2012,M1,M2}), the mathematical analysis of compressible fluid near vacuum regions presents significant challenges due to the inherent singularity and degeneracy of the governing equations. We briefly review several results concerning the global existence of solutions to the isentropic compressible Navier--Stokes system \eqref{1.2} with vacuum. In the pioneering work \cite{Li1998}, P.-L. Lions established the global existence of large weak solutions in $\mathbb{R}^3$ with $\gamma\geq\frac95$ via renormalized techniques and the effective viscous flux arguments. Subsequently, the restriction on the adiabatic index was extended in \cite{EF01,JZ01,JZ03}. Yet the question of global regularity and uniqueness for such weak solutions remains completely open. There are also some results regarding the global well-posedness of strong (or classical) solutions. Huang--Li--Xin \cite{HLX2012} proved the global existence and uniqueness of classical solutions in $\mathbb{R}^3$ with smooth initial data which are of small total energy but possibly large oscillations, where the far-field density could be either vacuum or non-vacuum. In \cite{HP24}, Hong et al. showed that the initial total energy could be large as long as the adiabatic exponent is sufficiently close to one. Based on some key {\it a priori} decay with rates (in large time) and a spatially weighted energy method, Li and Xin \cite{LX2019} dealt with the two-dimensional (2D for short) Cauchy problem with far-field vacuum under the condition that the initial total energy is suitably small. It should be mentioned that a central aspect of their studies in \cite{HLX2012,HP24,LX2019} is the derivation of the time-independent upper bound of the density. In fact, this coincides with the outcomes of the previous blow-up criterion \cite{HLX11}. A significant progress was recently
made in \cite{DM23}, where a unique global solution with vacuum in $\mathbb{T}^3$ was derived under a suitable scaling invariant smallness condition on the velocity. Motivated by \cite{LX2019-1}, where Lei and Xin discovered a scaling-invariant property of solutions under the scaling transformation
$$
\rho^k (x,t) = k^{\frac{1}{\gamma}} \rho\Big(k^{\frac{\gamma + 1}{2 \gamma}} x, k t\Big), \ \ u^k (x,t) = k^{\frac{\gamma - 1}{2 \gamma}} u\Big(k^{\frac{\gamma + 1}{2 \gamma}} x, k t\Big),
$$
Wen \cite{W2025} demonstrated the global existence of strong solutions to \eqref{1.2} in $\mathbb{R}^3$ provided that the scaling-invariant initial quantity
\begin{equation}\label{1.3}
\|\rho_0\|_{L^{\infty }}^3 \left( \|\sqrt{\rho_0} u_0\|_{L^2}^2 + \|\rho_0^{\gamma}\|_{L^1} \right) \left(\|\nabla u_0\|_{L^2}^2 + \|\rho_0^{\gamma}\|_{L^2}^2 \right)\left[1 + \|\rho_0\|_{L^{\infty }}^{3 + \gamma} \left( \|\sqrt{\rho_0} u_0\|_{L^2}^2 + \|\rho_0^{\gamma}\|_{L^1} \right)^2 \right]
\end{equation}
is properly small.

Let's turn our attention to the non-isentropic compressible Navier--Stokes equations \eqref{1.1}. As a system, \eqref{1.1} contains more complicated structures than the isentropic model. Indeed, it cannot be reduced to \eqref{1.2} even with zero temperature. Their distinctive features make analytical studies a great challenge but offer new opportunities. In 2004, Feireisl \cite{Fe2004} generalized the results of P.-L. Lions \cite{Li1998} to the case of compressible viscous heat-conductive fluids, and proved the global existence of {\it variational weak solutions} for arbitrary data in $N$-dimensional domains. As for strong solutions, it was shown in \cite{HL2018,WZ2017} the global existence and uniqueness of strong solutions in $\mathbb{R}^3$ with far-field non-vacuum and vacuum respectively which are of small energy but possibly large oscillations. It is worth noting that the basic energy inequality does not provide any dissipation estimates in the non-isentropic case, while some types of dissipation estimates were recovered in \cite{HL2018,WZ2017} with the aid of the entropy inequality and the conservation of mass respectively. By adopting the piecewise-estimate
method, Liang \cite{L21} established the global existence and decay rates of a unique strong solution with far-field vacuum when the initial energy is small. Under less regular data and weaker compatibility condition, Lai--Xu--Zhang \cite{LXZ2022} proved the global well-posedness and the exponential decay rates of classical solutions via using time-weighted estimates to avoid the presence of temperature initial-layer. Furthermore, Li \cite{Li2020} obtained the global existence of strong solutions to the 3D Cauchy problem under suitable smallness assumption on the initial data and the additional technical condition $2\mu>\lambda$. A peculiar aspect of his result is that the smallness assumption is imposed on a quantity that is scaling invariant with respect to the transformation
\begin{align*}
(\rho_{0\eta}(x),u_{0\eta}(x),\theta_{0\eta}(x))
=(\rho_{0}(\eta x),\eta u_{0}(\eta x),\eta^2\theta_{0}(\eta x)),\ \eta\neq0.
\end{align*}
More recently, a new scaling-invariant initial quantity
\begin{equation}\label{1.4}
\|\rho_0\|_{L^\infty}
\left(\|\rho_0\|_{L^3}+\|\rho_0\|_{L^\infty}^2\|\sqrt{\rho_0} u_0\|_{L^2}^2\right)
\left(\|\nabla u_0\|_{L^2}^2 +\|\rho_0\|_{L^\infty}\|\sqrt{\rho_0} \theta_0\|_{L^2}^2\right)
\end{equation}
was proposed in \cite{W2025}, where the author removed the restriction $2\mu>\lambda$ in \cite{Li2020} and established the global existence and uniqueness of strong solutions.

It should be emphasized that all results concerning strong solutions mentioned above require some smallness condition depending on the initial data and/or known parameters in the systems. However, from the viewpoint of partial differential equations (PDE), there are advantages of a parameter- and initial-data-independent smallness assumption in the study of global solutions to PDE. For example, such an assumption facilitates perturbation analysis and dynamic bifurcation research. On the other hand, for practical problems, the parameter- and initial-data-free smallness condition facilitates direct use of measured data and realistic physical constants in numerical simulations of PDE. Thus, this motivates us to investigate global strong solutions to the systems \eqref{1.2} and \eqref{1.1} with vacuum under some smallness conditions independent of initial data and model parameters.

\section{Main results}

Before stating our main results, we introduce the following notations that will be used throughout the paper.
\subsection{Notations}

\begin{itemize}
\item [\quad(i)] $\int f(x) \mathrm{d} x\triangleq \int_{\mathbb{R}^3} f(x) \mathrm{d} x$.
\item [\quad(ii)] For $p \in [1,\infty ]$ and $k \in \mathbb{Z}^+$, we denote the usual Sobolev spaces by
\begin{align*}
\begin{cases}
L^p\triangleq L^p(\mathbb{R}^3), \ \ W^{k,p}\triangleq\left\{u\in L^p ~|~ \nabla^iu\in L^p,\ i=1,\ldots,k\right\},\ \  H^k \triangleq W^{k,2},\\
D^{k,p}\triangleq\left\{u\in L^1_{\loc} ~|~ \nabla^k u\in L^p\right\},\ \
D^1\triangleq \{u \in L^6~|~\|\nabla u\|_{L^2} < \infty \}, \ \ D^k \triangleq D^{k,2}, \\
\|(\cdot, \ldots,\cdot)\|_{X} \triangleq \|\cdot\|_{X} + \ldots + \|\cdot \|_{X},\ \ \|\cdot\|_{X \cap Y} \triangleq \|\cdot\|_{X} + \|\cdot\|_{Y}.
\end{cases}
\end{align*}
\item [\quad(iii)] $F=(2\mu+\lambda)\divv u-P$ is the effective viscous flux.
\item [\quad(iv)]  $\omega=\nabla\times u$ denotes the vorticity.
\item [\quad(v)]  $\dot v=v_t+u\cdot\nabla v$ stands for the material derivative of $v$.
\end{itemize}

\subsection{The isentropic case}

System \eqref{1.2} is supplemented with the following initial data and the far-field behavior
\begin{gather}\label{2.1}
(\rho,\rho u)(x,0)=(\rho_0,m_0)(x), \ \  x\in \mathbb{R}^3,
\\ \label{2.2}
(\rho,u)\to(0,0) \ \ {\rm as} \ \  |x|\to\infty,  \ \  {\rm for}\ \ t\geq 0.
\end{gather}

Our main result for the isentropic compressible Navier--Stokes equations reads as follows:
\begin{Theorem}\label{thm2.1}
For any given $q\in(3, 6)$, assume that
\begin{equation}\label{2.3}
0 \leq\rho_0 \in L^{\gamma} \cap H^1 \cap W^{1,q},\ \
u_0 \in D^1,\ \ \sqrt{\rho_0}u_0 \in L^2.
\end{equation}
There exists a positive constant $\varepsilon$ independent of $\mu, \lambda, a, \gamma$, and the initial data such that if
\begin{equation}\label{2.4}
\mathbb{E}_0\triangleq \mu^{-5} \alpha^{6} \bar{\rho}^3 \mathbb{E}_1 (0) \mathbb{E}_2 (0) \leq \varepsilon,
\end{equation}
where
\begin{gather*}
\alpha \triangleq \max\{1 + a, \gamma\},\ \ \bar{\rho} \triangleq 1 + \|\rho_0\|_{L^{\infty }},\\
\mathbb{E}_1 (0) \triangleq \int \left(\frac{1}{2} \rho_0 |u_0|^2 + \frac{a}{\gamma - 1} \rho_0^{\gamma} \right) \mathrm{d} x,\\
\mathbb{E}_2 (0)
\triangleq \int \big[\mu |\nabla u_0|^2 + (\mu + \lambda)(\divv u_0)^2 + \left(2 \mu + \lambda \right)^{-1}P_0^2 \big] \mathrm{d} x,
\end{gather*}
then the problem \eqref{1.2}, \eqref{2.1}, and \eqref{2.2} have a unique global strong solution $(\rho, u)$ on $\mathbb{R}^3 \times (0,\infty)$ satisfying
\begin{equation}\label{2.5}
0\leq \rho(x,t)\leq \frac{3}{2}\bar\rho,\ \ {for \ all}\ \  (x,t)\in \mathbb{R}^3 \times(0, \infty)
\end{equation}
and
\begin{equation}\label{2.6}
\begin{cases}
\rho \in C([0, \infty);L^2) \cap L^\infty(0, \infty; L^{\gamma} \cap H^1 \cap W^{1,q}), \ \ \rho_t \in L^\infty(0, \infty;L^2\cap L^q),\\[1mm]
\sqrt \rho u \in C ([0, \infty);L^2),\ \ u \in L^\infty(0, \infty; D^1) \cap L^2 (0, \infty; D^1),\\[1mm]
\sqrt \rho \dot{u} \in L^2 (0, \infty; L^2), \ \ t \dot{u} \in L^2 (0, \infty; D^1),\ \ \sqrt{t \rho} \dot{u} \in L^\infty(0, \infty;L^2).
\end{cases}
\end{equation}
Moreover, there exists a positive constant $C$ depending on transport coefficients and initial norms but not on time such that, for any $t \geq 1$, $p \in [2,6]$, and $r \in (1,\infty )$,
\begin{equation}\label{2.7}
\|\nabla u\|_{L^p} \leq C t^{- 1 + \frac{1}{p}},\ \
\|P\|_{L^r} \leq C t^{- 1 + \frac{1}{r}},\ \
\|(\nabla F,\nabla \omega )\|_{L^2} \leq C t^{- 1}.
\end{equation}
\end{Theorem}

\begin{Remark}
It should be noted that our smallness condition \eqref{2.4} is independent of any known parameters in \eqref{1.2} and initial data, which is in sharp contrast to \cite{DM23,HLX2012,HP24,W2025} where they established global strong solutions under some smallness condition depending on the initial data and/or parameters. Moreover, there is no need to require the compatibility conditions on the initial data.
\end{Remark}

\begin{Remark}
It is worth mentioning that Wen \cite{W2025} established the global well-posedness of strong solutions to the problem \eqref{1.2}, \eqref{2.1}, and \eqref{2.2} provided the scaling invariant initial quantity \eqref{1.3} is suitably small. By contrast, our Theorem \ref{thm2.1} eliminates the term $\left[1 + \|\rho_0\|_{L^{\infty }}^{3 + \gamma} \left( \|\sqrt{\rho_0} u_0\|_{L^2}^2 + \|\rho_0^{\gamma}\|_{L^1} \right)^2 \right]$ in \eqref{1.3}.
\end{Remark}

We mainly use global {\it a priori} estimates (see Proposition \ref{pro4.1}) and Serrin-type blow-up criterion \eqref{3.1} to prove Theorem \ref{thm2.1}. It should be pointed out that the approaches in \cite{DM23,HLX2012,HP24,W2025} cannot be applied to the present framework directly, since their smallness assumptions are imposed on the initial data and known parameters in \eqref{1.2}. Hence, in order to derive the desired global \emph{a priori} estimates under the condition \eqref{2.4}, several new ideas are needed.

Inspired by \cite{W2025,LX2019}, we attempt to establish an estimate on the $L^{\infty}_t L^2_x$-norm of $\na u$, which necessitates an estimate on the $L^3_t L^3_x$-norm of the pressure (see \eqref{4.6}). To this end, we assume the {\it a priori hypothesis} \eqref{3.4}. Then, multiplying \eqref{4.7} by $\left( 2 \mu + \lambda\right)^{-1} P$ and exploiting the delicate energy estimates, we obtain the desired bound for $\|\na u\|_{L^{\infty}_t L^2_x}$.
In this process, the main difficulties stem from nonlinear and strongly coupled terms.
To overcome these obstacles, we introduce a new quantity $\mathbb{E}_0$ and establish coefficient- and initial-data-independent global {\it a priori} estimates with the aid of the effective viscous flux and vorticity provided $\mathbb{E}_0$ is suitably small.
The remaining challenge is to close the bootstrap argument. Compared with \cite{W2025}, closing the global {\it a priori} assumption for the density is considerably more challenging, since the smallness condition $\mathbb{E}_0$ defined in \eqref{2.4} excludes the term
$
\left[1 + \|\rho_0\|_{L^{\infty }}^{3 + \gamma} \left( \|\sqrt{\rho_0} u_0\|_{L^2}^2 + \|\rho_0^{\gamma}\|_{L^1} \right)^2 \right].
$
The novelty of our analysis is that the term
$
\int^{t_2}_{t_1} \|\na u\|_{L^2} \|\rho^{\gamma}\|_{L^6} {\rm d} \tau
$
can be bounded by $\mathbb{E}_1 (0)^{\frac{1}{3}} \mathbb{E}_2 (0)^{\frac{1}{3}}$ and $\frac{a \bar{\rho}^{\gamma}}{2 \mu + \lambda } \left( t_2 - t_1 \right)$, in contrast to the previous bound relying on $\mathbb{E}_1 (0)^{\frac{3}{4}} \mathbb{E}_2 (0)^{\frac{1}{4}}$ and $\frac{a \|\rho_0\|_{L^{\infty}}^{\gamma}}{2 \mu + \lambda } \left( t_2 - t_1 \right)$ (see \cite[Lemma 2.9]{W2025}).
This improved bound enables us to close the {\it a priori} assumption of the density provided $\mathbb{E}_0$ is small enough.

\subsection{The non-isentropic case}

We consider the Cauchy problem of system \eqref{1.1} with the initial data and far-field behavior
\begin{gather}\label{2.8}
(\rho,\rho u,\rho\theta )(x,0)=(\rho_0,m_0, n_0)(x), \ \  x\in \mathbb{R}^3,
\\ \label{2.9}
(\rho,u,\theta )\to(0,0,0) \ \ {\rm as} \ \  |x|\to\infty,  \ \  {\rm for}\ \ t\geq 0.
\end{gather}

Our main result for the non-isentropic compressible Navier--Stokes equations can be stated as follows:
\begin{Theorem}\label{thm2.2}
For any given $q\in(3, 6)$, assume that
\begin{equation}\label{2.10}
0 \leq \rho_0 \in L^{\frac{3}{2}} \cap H^1 \cap W^{1,q}, \ \ u_0 \in D^1,\ \ \sqrt{\rho_0}\theta_0\in L^2.
\end{equation}
There exists a positive constant $\varepsilon$ independent of $\mu, \lambda, \kappa, {\rm c_V}, \gamma, {\rm R}$, and the initial data such that if
\begin{equation}\label{2.11}
\mathbb{F}_0
\triangleq \mu^{-\frac{3}{2}} \Theta^{\frac{81}{2}} \bar{\rho}^{13} \mathbb{F}_1 (0) \mathbb{F}_2(0) \leq \varepsilon,
\end{equation}
where
\begin{gather*}
\Theta \triangleq \max\left\{1 + {\rm c_V}, \gamma, 1 + \mu^{-1}, 1 + \kappa^{-1}, 1 + \mu^{-1} |\lambda|\right\},\ \bar{\rho} \triangleq 1 + \|\rho_0\|_{L^1 \cap L^{\infty }},\\
\mathbb{F}_1 (0) \triangleq {\rm c_V} \|\rho_0 \theta _0\|_{L^1} + \|\sqrt{\rho_0} u_0\|_{L^2}^2,
\\
\mathbb{F}_2 (0) \triangleq \mu \|\nabla u_0\|_{L^2}^2 + \left(\mu + \lambda \right) \|\divv u_0\|_{L^2}^2 + \Psi {\rm c_V} \|\sqrt{\rho_0} \theta _0\|_{L^2}^2,\
\Psi \triangleq \mu^{-1} \Theta^3 \left(1 + \kappa \right) \bar{\rho},
\end{gather*}
then the problem \eqref{1.1}, \eqref{2.8}, and \eqref{2.9} have a unique global strong solution $(\rho, u, \theta)$ on $\mathbb{R}^3 \times (0,\infty)$ satisfying
\begin{equation}\label{2.12}
0\leq \rho(x,t)\leq \frac{3}{2}\bar\rho,\ \ {for \ all}\ \  (x,t)\in \mathbb{R}^3 \times(0, \infty)
\end{equation}
and
\begin{equation}\label{2.13}
\begin{cases}
\rho \in C([0,\infty);L^2) \cap L^\infty(0,\infty;L^{\frac{3}{2}} \cap H^1 \cap W^{1,q}), \ \rho_t \in L^\infty(0,\infty;L^2\cap L^q),\\[1mm]
(\sqrt \rho u, \sqrt \rho \theta ) \in C ([0,\infty);L^2), \ (u, \sqrt t \theta ) \in L^\infty(0,\infty; D^1) \cap L^2 (0,\infty; D^1 \cap D^{2}),\\[1mm]
(\sqrt \rho u_t, \sqrt{t \rho} \theta _t) \in L^2 (0,\infty; L^2), \ \sqrt t u \in L^\infty(0,\infty; D^{2}),\ 0 \leq \theta \in L^2 (0,\infty;D^1),\\[1mm]
(\sqrt{t \rho} u_t, t \sqrt \rho \theta _t, t \nabla^2 \theta ) \in L^\infty(0,\infty;L^2), \ t \theta _t \in L^2 (0,\infty; D^1), \ u \in L^1 (0,\infty; D^{2,q}).
\end{cases}
\end{equation}
Moreover, there exist two positive constants $C$ and $\sigma \in (0,1)$, which may depend on $\mu, \lambda, \kappa, {\rm c_V}, {\rm R}, q$, and the initial norms but not on time, such that
\begin{equation}\label{2.14}
\|\nabla \rho\|_{L^2} + \|\nabla \rho\|_{L^q}  \leq C,\quad \forall \ t\geq0,
\end{equation}
and for any $t \geq 1$,
\begin{equation}\label{2.15}
\|\rho_t\|_{L^2\cap L^q}^2 + \|(\sqrt \rho u,  \sqrt \rho u_t,\sqrt \rho \theta , \sqrt \rho \theta _t)\|_{L^2}^2 + \|(\nabla u, \nabla \theta )\|_{H^1}^2 \leq C e^{-\sigma t}.
\end{equation}
\end{Theorem}

\begin{Remark}
It should be noted that our smallness condition \eqref{2.11} is independent of any initial data and known parameters in \eqref{1.1}, which is in sharp contrast to \cite{L21,HL2018,WZ2017,Li2020,W2025,LXZ2022} where they established global strong solutions under some smallness condition depending on the initial data and/or parameters. Moreover, there is no need to require the compatibility conditions on the initial data.
\end{Remark}

\begin{Remark}
Compared with \cite{WZ2017}, where Wen and Zhu established the global well-posedness of strong solutions to the problem \eqref{1.1}, \eqref{2.8}, and \eqref{2.9} provided that either the shear viscosity coefficient $\mu$ and the heat-conductivity coefficient $\kappa$ are sufficiently large ($\kappa \sim \mu^{r_1}$ with $\frac{3}{2} < r_1 < 5$) or the initial mass is properly small, the present work removes the large assumption on $\kappa$.
\end{Remark}

The proof of Theorem \ref{thm2.2} relies primarily on global {\it a priori} estimates (see Proposition \ref{pro5.1}) and the blow-up criterion \eqref{3.3}.
It should be pointed that the methodologies developed in \cite{LXZ2022,L21,WZ2017,HL2018} cannot be straightforwardly adopted to our analysis, since their smallness constraints are imposed on the known parameters in \eqref{1.1} and initial data, making these approaches inconsistent with the scaling-invariant structure.
Moreover, the analyses in \cite{Li2020} and \cite{WZ2017} strongly depend on the condition $2 \mu > \lambda$ and the large-value assumption on $\mu$ and $\kappa$, respectively. Thus, we shall develop novel analytical observations and new techniques
to establish coefficient-independent {\it a priori} estimates under the scaling-invariant assumption \eqref{2.11}.

On the one hand, inspired by \cite{Li2020,W2025}, we derive a crucial dissipation estimate $\|\nabla u\|_{L^2_t L^2_x}^2$, which is required to be bounded by the initial data $\mathbb{F}_0$ in a suitable sense.
With the aid of the {\it a priori hypothesis} (\ref{5.1}), we obtain this desired estimate. Our subsequent task is to establish a bound for the $L^{\infty}_t L^2_x$-norm of $\na u$, which requires estimating the $L^2_t L^2_x$-norm of $\na \theta$.
By virtue of the delicate energy estimates and the smallness condition imposed on $\mathbb{F}_0$, we reach the targeted conclusion.
On the other hand, it is rather challenging to close the bootstrap argument.
Based on the estimates derived from Proposition \ref{pro5.1}, we establish the uniform bound for the density and achieve the desired control of $\mathbb{F} (t)$ provided that the quantity $\mathbb{F}_0$ is sufficiently small.
Specifically, to construct a complete scaling-invariant quantity for the {\it a priori} estimates and the bootstrap argument, we multiply (\ref{5.11}) by $\mu^{-1} \Theta^3 \left(1 + \kappa \right) \bar{\rho}$. In this procedure, it is quite challenging to prove that the crucial norm $\Theta \|\divv u\|_{L^1_t L^{\infty}_x} $ can be controlled by the initial data $\mathbb{F}_0$. To verify that $\Theta \|\divv u\|_{L^1_t L^{\infty}_x} $ is small in a suitable sense, we first adopt the effective viscous flux to estimate $\mu^{-1} \Theta \|P\|_{L^1_t L^{\infty}_x} $ and $\mu^{-1} \Theta \|F\|_{L^1_t L^{\infty}_x}$ as a substitute for the direct estimate. We then establish several time-weighted estimates for $\mathbb{F}_2 (t)$, $\mathbb{F}_3 (t)$, which are bounded by the initial data $\mathbb{F}_1 (0)$ and $\mathbb{F}_2 (0)$, respectively. Finally, by virtue of a novel piecewise estimation strategy for solutions, we successfully obtain the expected results.

The rest of the paper is organized as follows. In the next section, we recall some known facts and elementary inequalities that will be used later. Section \ref{sec4} is devoted to showing Theorem \ref{thm2.1}, while the proof of Theorem \ref{thm2.2} is carried out in the last section.

\section{Preliminaries}\label{sec3}

In this section, we recall some elementary inequalities and known facts.
We begin with the local well-posedness of strong solutions for the isentropic compressible Navier--Stokes equations in \cite{Hu2020} and a Serrin-type blow-up criterion in \cite{HLX11}.
\begin{Lemma}\label{lem3.1}
Assume that \eqref{2.3} holds. Then there exists a small time $T>0$ such that the  problem \eqref{1.2}, \eqref{2.1}, and \eqref{2.2} admit a unique strong solution $(\rho,u)$ satisfying \eqref{2.5} and \eqref{2.6} on $\mathbb{R}^3 \times (0, T]$.
Furthermore, if $T$ is the maximal time of existence for local solutions, then either $T=\infty$ or
\begin{align}\label{3.1}
\lim _{T_1 \rightarrow T}\left(\|\rho\|_{L^\infty(0, T_1;L^{\infty})}
+\|\sqrt{\rho}u\|_{L^s(0, T_1; L^r)}\right) = \infty,
\end{align}
with $r$ and $s$ satisfying
\begin{align}\label{3.2}
\frac{2}{s}+\frac{3}{r} \leq 1, \ \ s>1,\ \ 3<r \leq \infty.
\end{align}
\end{Lemma}

Next, we recall the local existence and uniqueness of strong solutions for the non-isentropic compressible Navier--Stokes equations in \cite{LZ2023} and a Serrin-type blow-up criterion in \cite{HLW13}.
\begin{Lemma}\label{lem3.2}
Assume that \eqref{2.10} holds. Then there exists a small time $T>0$ such that the problem \eqref{1.1}, \eqref{2.8}, and \eqref{2.9} admit a unique strong solution $(\rho,u,\theta)$ satisfying \eqref{2.12} and \eqref{2.13} on $\mathbb{R}^3\times (0, T]$. Furthermore, if $T$ is the maximal time of existence for local solutions, then either $T=\infty$ or
\begin{align}\label{3.3}
\lim _{T_1 \rightarrow T}\left(\|\divv u\|_{L^1(0, T_1;L^{\infty})}
+\|u\|_{L^s(0, T_1; L^r)}\right) = \infty,
\end{align}
with $r$ and $s$ satisfying \eqref{3.2}.
\end{Lemma}

Next, it deduces from (\ref{1.1})$_2$ that the effective viscous flux $F$ and the vorticity $\omega$ satisfy
\begin{equation}\label{3.4}
\Delta F = \divv (\rho \dot u),\ \  \mu \Delta \omega = \nabla \times(\rho \dot u).
\end{equation}
Applying the standard $L^p$-estimate to the system \eqref{3.4} leads to
\begin{Lemma}\label{lem3.3}(\cite{HLX2012})
Let $(\rho, u)$/$(\rho, u, \theta )$ be a strong solution to the problem \eqref{1.2}, \eqref{2.1}, and \eqref{2.2}/\eqref{1.1}, \eqref{2.8}, and \eqref{2.9}. Then, for any $p \in [2,6]$, there exists a positive constant $C=C(p)$ such that
\begin{gather}\label{3.5}
\|(\nabla F,\mu \nabla \omega )\|_{L^p} \leq C \|\rho \dot u\|_{L^p},
\\ \label{3.6}
\|(F, \mu \omega )\|_{L^p} \leq C \|\rho \dot u\|_{L^2}^{\frac{3p - 6}{2p}} \|(\mu \nabla u, (\mu + \lambda ) \divv u, P)\|_{L^2}^{\frac{6 - p}{2p}},
\\ \label{3.7}
\|\nabla u\|_{L^p}
\leq C \mu^{\frac{6 - 3p}{2p}} \|\nabla u\|_{L^2}^{\frac{6 - p}{2p}} \|\rho \dot{u}\|_{L^2}^{\frac{3 p -6}{2p}} + C \left( 2 \mu + \lambda \right)^{-1} \left( \|P\|_{L^p} + \|P\|_{L^2}^{\frac{6 - p}{2p}} \|\rho \dot{u}\|_{L^2}^{\frac{3 p -6}{2p}} \right).
\end{gather}
\end{Lemma}

Next, the following Gagliardo--Nirenberg inequality (see \cite[Theorem 12.87]{book17}) will be used.
\begin{Lemma}\label{lem3.4}
 If $0\leq m,j\leq l$, then it has
$$
\|\nabla^{j}f\|_{L^p} \leq C\|\nabla^m f\|_{L^q}^{1-\theta} \|\nabla^l f\|_{L^r}^\theta,
$$
where $0\leq \theta\leq 1$ $($when $p=\infty$, it requires $0<\theta<1)$   and $j$ satisfies
$$
\frac{j}{3} - \frac{1}{p} = \left( \frac{m}{3} - \frac{1}{q} \right) (1-\theta) + \left(\frac{l}{3} - \frac{1}{r} \right) \theta.
$$
\end{Lemma}

Finally, the following Zlotnik's inequality (see \cite[Lemma 1.3]{Z2000}) play a crucial role to achieve the uniform upper bound of the density.

\begin{Lemma}\label{lem3.5}
Suppose that the function $y$ satisfies
$$
y'(t) = g (y) + b'(t) \ \ on \ \ [0,T], \ \ y(0) = y^0,
$$
with $g \in C (\mathbb{R})$ and $y,b \in W^{1,1} (0,T)$. If $g(\infty ) = - \infty $ and
\begin{equation}\label{3.8}
b (t_2) - b(t_1) \leq N_0 + N_1 (t_2 - t_1),
\end{equation}
for all $0 \leq t_1 < t_2 \leq T$ with some $N_0 \geq 0$ and $N_1 \geq 0$, then $$
y(t) \leq \max \left\{y^0, \alpha_0 \right\} + N_0 < \infty\ \ on \ \ [0,T],
$$
where $\alpha_0$ is a constant such that
\begin{equation}\label{3.9}
g (\alpha) \leq - N_1 \ \ for \ \ \alpha \geq \alpha_0.
\end{equation}
\end{Lemma}

\section{Proof of Theorem \ref{thm2.1}}\label{sec4}

This section aims to proving Theorem \ref{thm2.1}. To this end, we first establish some global {\it a priori} estimates under smallness condition on the coefficient-coupled scaling invariant initial quantity, and then extend the local solution to be the global one.
Let $(\rho,u)$ be the strong solution of \eqref{1.2}, \eqref{2.1}, and \eqref{2.2} on $\mathbb{R}^3 \times (0,T]$ and set
\begin{gather*}
\mathbb{E}_1 (t)
\triangleq \int \left(\frac{1}{2} \rho |u|^2 + \frac{a}{\gamma - 1} \rho^{\gamma} \right) \mathrm{d} x,
\\
\mathbb{E}_2 (t)
\triangleq \int \big[\mu |\nabla u|^2 + (\mu + \lambda )(\divv u)^2 + \left(2 \mu + \lambda \right)^{-1}P^2 \big] \mathrm{d} x,
\\
\mathbb{E} (t) \triangleq \mu^{-5} \alpha^{6} \bar{\rho}^{3} \mathbb{E}_1 (t) \mathbb{E}_2 (t).
\end{gather*}
Here and hereafter, we denote by $K, \varepsilon, K_i$, and $\varepsilon_i$ ($i=1, 2, \ldots$) the various positive constants, which are independent of the known parameters in \eqref{1.2}, the initial data, and $T$. Sometimes we write $K(\xi)$ to emphasize the dependence on $\xi$.

\begin{Proposition}\label{pro4.1}
Let the conditions in Theorem \ref{thm2.1} be satisfied.
There exists a positive constant $\varepsilon$ such that if
\begin{equation}\label{4.1}
\sup_{0 \leq t \leq T}\|\rho \|_{L^\infty } \leq 2 \bar \rho \ \ and \ \ \sup_{0 \leq t \leq T} \mathbb{E} (t) \leq \mathbb{E}_0^{\frac{1}{2}},
\end{equation}
then
\begin{gather}\label{4.3}
\sup_{0 \leq t \leq T} \mathbb{E}_1 (t) + \int^T_0 \left[ \mu \|\nabla u\|_{L^2}^2 + (\mu + \lambda ) \|\divv u\|_{L^2}^2 \right] \mathrm{d} t \leq \mathbb{E}_1 (0),
\\ \label{4.4}
\sup_{0 \leq t \leq T} \mathbb{E}_2 (t) + \int^T_0 \left[\|\sqrt{\rho} \dot{u}\|_{L^2}^2 + (2 \mu + \lambda )^{-2} \|P\|_{L^3}^3 \right] \mathrm{d} t \leq K \mathbb{E}_2 (0),\\
\label{4.2}
\sup_{0 \leq t \leq T}\|\rho \|_{L^\infty } \leq \frac{3}{2} \bar \rho, \ \ \sup_{0 \leq t \leq T}\mathbb{E} (t) \leq \frac{1}{2} \mathbb{E}_0^{\frac{1}{2}},
\end{gather}
provided that $\mathbb{E}_0 \leq \varepsilon.$
\end{Proposition}
\begin{proof}
{\bf Step I. Proof of \eqref{4.3}.}
The desired \eqref{4.3} follows from the standard elementary energy estimate.

{\bf Step II. Proof of \eqref{4.4}.} Taking the inner product of \eqref{1.2}$_2$ with $u_t$ in $L^2$, one has that
\begin{align}\label{4.5}
& \frac{1}{2} \frac{\mathrm{d}}{\mathrm{d} t} \int \left[ \mu |\nabla u|^2 + (\mu + \lambda )(\divv u)^2 \right] \mathrm{d} x + \int \rho |\dot{u}|^2 \mathrm{d} x\nonumber\\
& \quad = \int P \divv u_t \mathrm{d} x + \int \rho u \cdot \nabla u \cdot \dot{u} \mathrm{d} x\nonumber\\
& \quad = \frac{\mathrm{d}}{\mathrm{d} t} \int P \divv u \mathrm{d} x - \frac{1}{2 (2 \mu + \lambda )} \frac{\mathrm{d}}{\mathrm{d} t} \int P^2 \mathrm{d} x + \int \rho u \cdot \nabla u \cdot \dot{u} \mathrm{d} x \notag\\
& \qquad - \frac{1}{2\mu + \lambda } \int P u \cdot \nabla F \mathrm{d} x + \frac{\gamma - 1}{ \left(2 \mu + \lambda \right)} \int P \divv u F \mathrm{d} x\nonumber\\
& \quad \triangleq \frac{\mathrm{d}}{\mathrm{d} t} \int P \divv u \mathrm{d} x - \frac{1}{2 (2 \mu + \lambda )} \frac{\mathrm{d}}{\mathrm{d} t} \int P^2 \mathrm{d} x + \sum^3_{i = 1} I_i.
\end{align}
It deduces from (\ref{3.5}) with $p = 2$, (\ref{4.1}), and (\ref{3.7}) with $p = 3$ that
\begin{align*}
I_1
& \leq K \bar{\rho}^{\frac{1}{2}} \|\sqrt{\rho} \dot{u}\|_{L^2} \|\nabla u\|_{L^3} \|u\|_{L^6} \\
& \leq \frac{1}{8} \|\sqrt{\rho} \dot{u}\|_{L^2}^2 + K \mu^{-1} \bar{\rho} \|\nabla u\|_{L^2}^3 \|\rho \dot u\|_{L^2} + K \left(2 \mu + \lambda \right)^{-2} \bar{\rho} \|\nabla u\|_{L^2}^2 \left(\|P\|_{L^3}^2 + \|\rho \dot u\|_{L^2} \|P\|_{L^2} \right)\\
& \leq K \mu^{-5} \bar{\rho}^3 \|(\mu^{\frac{1}{2}} \nabla u, \left(2\mu + \lambda \right)^{-\frac{1}{2}} P)\|_{L^2}^2 \|\mu^{\frac{1}{2}}\nabla u\|_{L^2}^4 + \frac{1}{4} \|\sqrt{\rho} \dot{u}\|_{L^2}^2 + \frac{1}{16 (2 \mu + \lambda )^2} \|P\|_{L^3}^3,\\
I_2 + I_3
& \leq K \left(2 \mu + \lambda \right)^{-1} \gamma \|\nabla F\|_{L^2} \|P\|_{L^3} \|\nabla u\|_{L^2} \\
& \leq \frac{1}{4} \|\sqrt{\rho} \dot{u}\|_{L^2}^2 + K \left(2 \mu + \lambda \right)^{-2} \gamma^2 \bar{\rho} \|\nabla u\|_{L^2}^2 \|P\|_{L^3}^2\\
& \leq \frac{1}{4} \|\sqrt{\rho} \dot{u}\|_{L^2}^2 + \frac{1}{16 (2 \mu + \lambda )^2} \|P\|_{L^3}^3 + K \mu^{-5} \alpha^6 \bar{\rho}^3 \|\mu^{\frac{1}{2}} \nabla u\|_{L^2}^6.
\end{align*}
Inserting the above estimates into (\ref{4.5}), we get that
\begin{align}\label{4.6}
& \frac{\mathrm{d}}{\mathrm{d} t} \int \big[ \mu |\nabla u|^2 + (\mu + \lambda )(\divv u)^2 + \left( 2 \mu + \lambda \right)^{-1}|P|^2 \big] \mathrm{d} x + \int \rho |\dot{u}|^2 \mathrm{d} x\notag \\
& \quad \leq 2 \frac{\mathrm{d}}{\mathrm{d} t} \int P \divv u \mathrm{d} x + \frac{\|P\|_{L^3}^3}{4 (2 \mu + \lambda )^2} + K \mu^{-5} \alpha^6 \bar{\rho}^3 \mathbb{E}_2 (t) \|\sqrt{\mu}\nabla u\|_{L^2}^4.
\end{align}

Next, for positive constant $K_1 > 1 $, one has that
$$
2 \int P \divv u \mathrm{d} x
\leq \frac{2\mu + \lambda }{4} \|\divv u\|_{L^2}^2 + K_1 \frac{ \|P\|_{L^2}^2}{2 \mu + \lambda } \leq \frac{1}{4} \mathbb{E}_2 (t) + K_1 \frac{ \|P\|_{L^2}^2}{2 \mu + \lambda }.
$$
Moreover, we rewrite (\ref{1.2})$_1$ as
\begin{equation}\label{4.7}
P_t + \gamma P \divv u + u \cdot \nabla P = 0.
\end{equation}
Multiplying (\ref{4.7}) by $2 K_1 P \left(2 \mu + \lambda \right)^{-1}$ in $L^2$ and using the fact
$\divv u = \frac{F + P }{2 \mu + \lambda }$ gives
\begin{align*}
& \frac{K_1}{2 \mu + \lambda } \frac{\mathrm{d}}{\mathrm{d} t} \int P^2 \mathrm{d} x + \frac{K_1 \left( 2 \gamma - 1\right)}{\left(2 \mu + \lambda \right)^2} \int P^3 \mathrm{d} x\notag\\
& \quad = - \frac{ K_1 \left(2 \gamma - 1 \right)}{\left(2 \mu + \lambda \right)^2} \int P^2 F \mathrm{d} x \leq \frac{K_1 \left(2 \gamma - 1 \right)}{\left(2 \mu + \lambda \right)^2} \|P\|_{L^3}^2 \|F\|_{L^3}\notag\\
& \quad \leq \frac{K_1 \left(2 \gamma - 1 \right)}{4 (2 \mu + \lambda )^2} \|P\|_{L^3}^3 + K \frac{\left(2 \gamma - 1 \right)}{(2 \mu + \lambda )^2} \|((2 \mu + \lambda )\divv u, P)\|_{L^2}^{\frac{3}{2}} \|\rho \dot u\|_{L^2}^{\frac{3}{2}}\notag\\
& \quad \leq \frac{K_1 \left(2 \gamma - 1 \right)}{4 (2 \mu + \lambda )^2} \|P\|_{L^3}^3 + \frac{1}{4} \|\sqrt{\rho} \dot{u}\|_{L^2}^2 + K \frac{\alpha^{4} \bar{\rho}^3}{\left( 2 \mu + \lambda \right)^8} \|((2 \mu + \lambda )\divv u, P)\|_{L^2}^6,
\end{align*}
which together with \eqref{4.6} indicates that
\begin{align*}
& \frac{\mathrm{d}}{\mathrm{d} t} \left[\mu \|\nabla u\|_{L^2}^2 + \left(\mu + \lambda \right)\|\divv u\|_{L^2}^2 + \frac{K_1 + 1}{2 \mu + \lambda } \|P\|_{L^2}^2 \right] + \frac{3}{4} \|\sqrt{\rho} \dot{u}\|_{L^2}^2 + \frac{K_1 \left( 2 \gamma - 1\right)}{2 \left(2 \mu + \lambda \right)^2} \|P\|_{L^3}^3\\
& \quad \leq 2 \frac{\mathrm{d}}{\mathrm{d} t} \int P \divv u \mathrm{d} x + K \left( 2 \mu + \lambda \right)^{-3} \mu^{-5} \alpha^4 \bar{\rho}^3 \|P\|_{L^2}^6 + K \mu^{-5} \alpha^6 \bar{\rho}^3 \mathbb{E}_2 (t) \|(\sqrt{\mu}\nabla u, \sqrt{\mu + \lambda} \divv u)\|_{L^2}^4.
\end{align*}
Integrating the above inequality over $(0,T)$ with respect to $t$ leads to
\begin{align*}
& \sup_{0 \leq t \leq T} \mathbb{E}_2 (t) + \frac{3}{4}\int^T_0 \|\sqrt{\rho} \dot{u}\|_{L^2}^2 \mathrm{d} t + \frac{K_1 \left( 2 \gamma - 1\right)}{2 \left(2 \mu + \lambda \right)^2} \int^T_0 \|P\|_{L^3}^3 \mathrm{d} t\\
& \quad \leq K \mathbb{E}_2 (0) + \frac{1}{4} \sup_{0 \leq t \leq T} \mathbb{E}_2 (t) + K \left( 2 \mu + \lambda \right)^{-3} \mu^{-5} \alpha^4 \bar{\rho}^3 \int^T_0 \|P\|_{L^2}^6 \mathrm{d} t\\
& \qquad + K \mu^{-5} \alpha^6 \bar{\rho}^3 \int^T_0 \mathbb{E}_2 (t) \|(\sqrt{\mu}\nabla u, \sqrt{\mu + \lambda} \divv u)\|_{L^2}^4 \mathrm{d} t\\
& \quad \leq K \mathbb{E}_2 (0) + \frac{1}{2} \sup_{0 \leq t \leq T} \mathbb{E}_2 (t) + \frac{1}{4 \left(2 \mu + \lambda \right)^2} \int^T_0 \|P\|_{L^3}^3 \mathrm{d} t,
\end{align*}
where one has used
\begin{align*}
& K \left( 2 \mu + \lambda \right)^{-3} \mu^{-5} \alpha^4 \bar{\rho}^3 \int^T_0 \|P\|_{L^2}^6 \mathrm{d} t + K \mu^{-5} \alpha^6 \bar{\rho}^3 \int^T_0 \mathbb{E}_2 (t) \|(\sqrt{\mu}\nabla u, \sqrt{\mu + \lambda} \divv u)\|_{L^2}^4 \mathrm{d} t\\
& \quad \leq K \left( 2 \mu + \lambda \right)^{-3} \mu^{-5} \alpha^4 \bar{\rho}^3 \int^T_0 \|P\|_{L^2}^2 \|P\|_{L^1} \|P\|_{L^3}^3 \mathrm{d} t + K \mu^{-5} \alpha^6 \bar{\rho}^3 \sup_{0 \leq t \leq T}\big(\mathbb{E}_1 (t) \mathbb{E}_2(t)^2\big)\\
& \quad  \leq K \mu^{-5} \alpha^5 \bar{\rho}^3 \sup_{0 \leq t \leq T}\big(\mathbb{E}_1(t) \mathbb{E}_2(t)\big)\frac{1}{(2 \mu + \lambda )^2} \int^T_0 \|P\|_{L^3}^3 \mathrm{d} t + K \mathbb{E}_0^{\frac{1}{2}} \sup_{0 \leq t \leq T} \mathbb{E}_2 (t)\\
&  \quad \leq \frac{K_2 \mathbb{E}_0^{\frac{1}{2}}}{(2 \mu + \lambda )^2} \int^T_0 \|P\|_{L^3}^3 \mathrm{d} t + K_3 \mathbb{E}_0^{\frac{1}{2}} \sup_{0 \leq t \leq T} \mathbb{E}_2 (t) \leq \frac{1}{4} \sup_{0 \leq t \leq T} \mathbb{E}_2 (t) + \frac{1}{4 (2 \mu + \lambda )^2} \int^T_0 \|P\|_{L^3}^3 \mathrm{d} t,
\end{align*}
provided that
$$
\mathbb{E}_0 \leq \varepsilon_1 \triangleq \min \left\{1, (4 K_2)^{-2}, (4 K_3)^{-2}\right\}.
$$

{\bf Step III. Proof of (\ref{4.2}).}
For any given $(x,t) \in \mathbb{R}^3 \times [0,T]$ and $\delta>0$, let
$$
\rho^{\delta} (y,s) = \rho (y,s) + \delta \exp\left\{- \int^s_0 \divv u \left( X(\tau; x,t), \tau \right)\mathrm{d} \tau \right\} > 0.
$$
Here $X(s;x,t)$ is given by
$$
\begin{cases}
\frac{\mathrm{d}}{\mathrm{d} t} X (s;x,t) = u \left(X (s;x,t),s \right), \ \ 0 \leq s < t,\\
X(t;x,t) = x.
\end{cases}
$$
Thus, \eqref{1.2}$_1$ yields that
$$
\frac{\mathrm{d}}{\mathrm{d} s} \rho^{\delta} \left(X (s;x,t),s \right) + \rho^{\delta} \left(X (s;x,t),s \right) \divv u \left(X (s;x,t),s \right) = 0.
$$
This leads to
$$
Y'(s) = g(s) + b'(s),
$$
where
$$
Y (s) = \ln \rho^{\delta} \left(X (s;x,t),s \right), \ \ g(s) = - \frac{a \rho^{\gamma} \left(X (s;x,t),s \right)}{2 \mu + \lambda },
$$
$$
b(s) = - \frac{1}{2\mu + \lambda } \int^s_0 F \left(X (\tau;x,t),\tau \right) \mathrm{d} \tau.
$$

For the effective viscous flux $F$, one deduces from (\ref{3.4})$_1$ that
$$
F \left(X (t;x,\tau),\tau \right) = - \frac{\mathrm{d}}{\mathrm{d}
\tau} \left[(- \Delta)^{-1} \divv (\rho u) \right] + [u_i,\mathcal{R}_{ij} ] (\rho u_j)
$$
with $[u_i, \mathcal{R}_{ij}] = u_i \mathcal{R}_{ij} - \mathcal{R}_{ij} u_i$ and $\mathcal{R}_{ij} = \partial_i \left(- \Delta\right)^{-1} \partial_j$, which gives
\begin{align}
b(t_2) - b(t_1)
& = \frac{1}{2 \mu + \lambda } \left[(- \Delta)^{-1} \mathrm{div} (\rho u) (t_2) - (- \Delta)^{-1} \divv (\rho u) (t_1) \right] - \frac{1}{2 \mu + \lambda } \int^{t_2}_{t_1} [u_i,\mathcal{R}_{ij} ] (\rho u_j) \mathrm{d} \tau \notag\\
& \leq \frac{2}{2 \mu + \lambda } \sup_{t_1 \leq t \leq t_2} \| (- \Delta)^{-1} \divv (\rho u)\|_{L^{\infty }} - \frac{1}{2 \mu + \lambda } \int^{t_2}_{t_1} [u_i,\mathcal{R}_{ij} ] (\rho u_j) \mathrm{d} \tau \triangleq J_1 + J_2.\label{4.8}
\end{align}
Now we estimate $J_1$ and $J_2$ term by term
\begin{align}\label{4.9}
J_1
& \leq K \mu^{-1} \|(-\Delta)^{-1} \divv (\rho u)\|_{L^6}^{\frac{1}{2}} \|\nabla (-\Delta)^{-1} \divv (\rho u)\|_{L^6}^{\frac{1}{2}} \leq K \mu^{-1} \|\rho u\|_{L^2}^{\frac{1}{2}} \|\rho u\|_{L^6}^{\frac{1}{2}}\notag \\
& \leq K \mu^{- \frac{5}{4}} \bar{\rho}^{\frac{3}{4}} \|\sqrt \rho u\|_{L^2}^{\frac{1}{2}} \left( \mu \|\nabla u\|_{L^2}^2 \right)^{\frac{1}{4}} \leq K \mathbb{E}_0^{\frac{1}{8}}.
\end{align}
Before estimating $J_2$, we deduce from (\ref{3.7}) with $p = 6$ that
$$
\|\nabla u\|_{L^6}
\leq K \mu^{-1} \bar{\rho}^{\frac{1}{2}} \|\sqrt{\rho} \dot{u}\|_{L^2} + K \mu^{-1} a \|\rho^{\gamma}\|_{L^6},
$$
which combined with  the commutator estimates and Lemma \ref{lem3.4} leads to
$$
\begin{aligned}
\|[u_i,\mathcal{R}_{ij} ] (\rho u_j)\|_{L^\infty }
& \leq K \|[u_i,\mathcal{R}_{ij} ] (\rho u_j)\|_{L^3}^{\frac{1}{5}} \|\nabla [u_i,\mathcal{R}_{ij} ] (\rho u_j)\|_{L^4}^{\frac{4}{5}}\\
& \leq K \|u\|_{L^6}^{\frac{1}{5}} \|\rho u\|_{L^6}^{\frac{1}{5}} \|\nabla u\|_{L^6}^{\frac{4}{5}} \|\rho u\|_{L^{12}}^{\frac{4}{5}}\\
& \leq K \bar{\rho} \|u\|_{L^6}^{\frac{2}{5}} \|\nabla u\|_{L^6}^{\frac{4}{5}} \Big( \|u\|_{L^6}^{\frac{3}{4}} \|\nabla u\|_{L^6}^{\frac{1}{4}}\Big)^{\frac{4}{5}} \leq K \bar{\rho} \|\nabla u\|_{L^2} \|\nabla u\|_{L^6}\\
& \leq K \mu^{-1} \bar{\rho}^{\frac{3}{2}} \|\nabla u\|_{L^2} \|\sqrt{\rho} \dot{u}\|_{L^2} + K \mu^{-1} a \bar{\rho} \|\nabla u\|_{L^2} \|\rho^{\gamma}\|_{L^6}. \\
\end{aligned}
$$
This implies that
\begin{align*}
J_2
& \leq K \left(2 \mu + \lambda \right)^{-1} \mu^{-\frac{3}{2}} \bar{\rho} \left(\int^{t_2}_{t_1} \|\mu^{\frac{1}{2}}\nabla u\|_{L^2}^2 \mathrm{d} \tau \right)^{\frac{1}{2}} \left[\int^{t_2}_{t_1} \left(\bar{\rho} \|\sqrt{\rho} \dot{u}\|_{L^2}^2 + a^2 \|\rho^{\gamma}\|_{L^6}^2 \right)\mathrm{d} \tau \right]^{\frac{1}{2}}\notag\\
& \leq K \mathbb{E}_0^{\frac{1}{4}} + K \mu^{-\frac{3}{2}} \left(2 \mu + \lambda \right)^{-1} a \bar{\rho} \left( \int^{t_2}_{t_1} \|\mu^{\frac{1}{2}}\nabla u\|_{L^2}^2 \mathrm{d} \tau \right)^{\frac{1}{2}} \left(\int^{t_2}_{t_1} \|\rho^{\gamma}\|_{L^6}^2\mathrm{d} \tau \right)^{\frac{1}{2}}\notag\\
& \leq K \mathbb{E}_0^{\frac{1}{4}}  + K \mu^{-\frac{3}{2}} \left(2 \mu + \lambda \right)^{-1} a \bar{\rho}^{\frac{2 + \gamma}{2}} \sup_{t_1 \leq t \leq t_2} \mathbb{E}_2(t)^{\frac{1}{6}} \left[\int^{t_2}_{t_1} \left(\|\mu^{\frac{1}{2}}\nabla u\|_{L^2}^2 \right)^{\frac{2}{3}} \mathrm{d} \tau \right]^{\frac{1}{2}} \left(\int^{t_2}_{t_1} \|\rho^{\gamma}\|_{L^3}\mathrm{d} \tau \right)^{\frac{1}{2}}\notag\\
& \leq K \mathbb{E}_0^{\frac{1}{4}}  + K \mu^{-\frac{3}{2}} \left(2 \mu + \lambda \right)^{-\frac{2}{3}} a^{\frac{1}{2}} \bar{\rho}^{\frac{2 + \gamma}{2}}  \left( t_2 - t_1 \right)^{\frac{1}{2}} \sup_{t_1 \leq t \leq t_2} \mathbb{E}_2(t)^{\frac{1}{6}} \left(\int^{t_2}_{t_1} \|\mu^{\frac{1}{2}}\nabla u\|_{L^2}^2 \mathrm{d} \tau \right)^{\frac{1}{3}}\notag \\
& \quad \times\left[\frac{1}{(2 \mu + \lambda )^2} \int^{t_2}_{t_1} \|P\|_{L^3}^3 \mathrm{d} \tau \right]^{\frac{1}{6}} \leq K \mathbb{E}_0^{\frac{1}{4}} + K \mathbb{E}_0^{\frac{1}{3}} + \frac{a \bar{\rho}^{\gamma}}{2 \mu + \lambda } \left( t_2 - t_1 \right),
\end{align*}
which together with \eqref{4.8} and \eqref{4.9} shows that
$$
b(t_2) - b(t_1)
\leq K\mathbb{E}_0^{\frac{1}{8}} \Big(1 + \mathbb{E}_0^{\frac{7}{24}} \Big) + \frac{a \bar{\rho}^{\gamma}}{2 \mu + \lambda } \left( t_2 - t_1 \right) \leq K \mathbb{E}_0^{\frac{1}{8}} + \frac{a \bar{\rho}^{\gamma}}{2 \mu + \lambda } \left( t_2 - t_1 \right) \leq N_0 + N_1 \left( t_2 - t_1 \right).
$$
Here
$$
N_0 = K \mathbb{E}_0^{\frac{1}{8}}, \ \ N_1 = \frac{a \bar{\rho}^{\gamma}}{2 \mu + \lambda }.
$$

Recalling that
$$
g = g (Y) \triangleq - \left(e^{Y} - \delta \exp\left\{- \int^s_0 \divv u \left(X(\tau;x,t), \tau \right) \mathrm{d} \tau \right\} \right)^{\gamma} \left( 2 \mu + \lambda \right)^{-1},
$$
which implies that $g(\infty ) = - \infty $. Assume that
$$
\bar{Y}_{\delta} \triangleq \ln \left\{ \bar{\rho} + \delta \exp\left\{\int^{t_2}_{t_1} \|\divv u(\cdot,\tau)\|_{L^{\infty }} \mathrm{d} \tau \right\} \right\}.
$$
Let $Y \geq \bar{Y}_{\delta}$, then
$$
\begin{aligned}
g(Y)
& = - \left(e^{Y} - \delta \exp\left\{- \int^s_0 \divv u \left(X(\tau;x,t), \tau \right) \mathrm{d} \tau \right\} \right)^{\gamma} \left( 2 \mu + \lambda \right)^{-1}\\
& \leq - \left(e^{Y} - \delta \exp\left\{\int^{t_2}_{t_1} \|\mathrm{div} u (\cdot,\tau)\|_{L^{\infty }} \mathrm{d} \tau \right\} \right)^{\gamma} \left( 2 \mu + \lambda \right)^{-1} \leq - N_1.
\end{aligned}
$$
Consequently, one gets from Lemma \ref{lem3.5} that
$$
\rho^{\delta} (x, s) \leq \max\left\{ \bar{\rho} + \delta, \exp\{\bar{Y}_{\delta} \}\right\} \exp\{N_0\}.
$$
Letting $\delta \rightarrow 0$, one has
$$
\rho(x,s) \leq \bar{\rho} \exp\{N_0\} \leq \bar{\rho} \exp\left\{K_4 \mathbb{E}_0^{\frac{1}{8}}\right\} \leq \frac{3}{2} \bar{\rho},
$$
provided that
$$
\mathbb{E}_0 \leq \varepsilon_2 \triangleq \min \left\{\varepsilon_1,  \left(\frac{ \ln \frac{3}{2}}{K_4} \right)^8\right\}.
$$

Finally, it follows from (\ref{4.3}) and (\ref{4.4}) that
$$
\mathbb{E} (t) \leq K_5 \mathbb{E}_0 \leq \frac{1}{2} \mathbb{E}_0^{\frac{1}{2}},
$$
provided that
$$
\mathbb{E}_0 \leq \varepsilon \triangleq \min \left\{\varepsilon_2, (2 K_5)^{-2} \right\}.
$$
The proof of Proposition \ref{pro4.1} is complete.
\end{proof}

With the global {\it a priori} estimates in Proposition \ref{pro4.1} at hand, we can give the proof of Theorem \ref{thm2.1}.

{\bf Proof of Theorem \ref{thm2.1}}. The global existence of solutions can be obtained from Proposition \ref{pro4.1} and blow-up criterion \eqref{3.1}.
Indeed, taking $s=4$ and $r=6$ in \eqref{3.2} and using \eqref{4.3}--\eqref{4.2}, there exists a positive constant $C$ independent of $T$ such that
$$
\sup_{0\leq t\leq T} \|\rho\|_{L^\infty} + \int_{0}^T \|\sqrt{\rho}u\|_{L^6}^4 {\rm d} t \le \sup_{0\leq t\leq T} \|\rho\|_{L^\infty} + C \sup_{0\leq t\leq T} \|\rho\|_{L^\infty}^{\frac12} \sup_{0\leq t\leq T} \|\nabla u\|_{L^2}^{2}
\int_{0}^T \|\nabla u\|_{L^2}^{2}{\rm d}t \le C,
$$
which combined with \eqref{3.1} implies that $T=\infty$. Moreover, the proof of decay rate  \eqref{2.7} is similar to that in \cite{LX2019}, and we omit the details for simplicity.
This completes the proof of Theorem \ref{thm2.1}.
\hfill$\square$

\section{Proof of Theorem \ref{thm2.2}}\label{sec5}

In this section we give the proof of Theorem \ref{thm2.2} with the aid of global {\it a priori} estimates and Serrin-type blow-up criterion.
To this end, we set
\begin{gather}
\mathbb{F}_1 (t) \triangleq {\rm c_V} \|\rho \theta \|_{L^1} + \|\sqrt \rho u\|_{L^2}^2,\notag
\\
\mathbb{F}_2 (t)
\triangleq \mu \|\nabla u\|_{L^2}^2 + (\mu + \lambda ) \|\divv u\|_{L^2}^2 + \Psi {\rm c_V} \|\sqrt \rho \theta \|_{L^2}^2,\notag
\\
\mathbb{F}_3 (t)
\triangleq \Psi \kappa \|\nabla \theta\|_{L^2}^2 + \|\sqrt{\rho} \dot{u}\|_{L^2}^2,\notag
\\
\mathbb{F} (t) \triangleq \mu^{-\frac{3}{2}} \Theta^{\frac{81}{2}} \bar{\rho}^{13} \mathbb{F}_1(t) \mathbb{F}_2 (t).\notag
\end{gather}
From now on, let $(\rho,u,\theta )$ be the strong solution of \eqref{1.1}, \eqref{2.8}, and \eqref{2.9} on $\mathbb{R}^3 \times (0,T]$. Then we have the following global {\it a priori} estimates.
\begin{Proposition}\label{pro5.1}
Let the conditions in Theorem \ref{thm2.2} be satisfied.
There exists a positive constant $\varepsilon$ such that if
\begin{equation}\label{5.1}
\sup_{0 \leq t \leq T} \mathbb{F} (t) \leq \mathbb{F}_0^{\frac{1}{2}} \ \ and \ \ \Theta \int^T_0 \| \divv u\|_{L^{\infty}} \mathrm{d} t \leq \mathbb{F}_1 (0)^{\frac{1}{32}},
\end{equation}
then
\begin{gather}\label{5.2}
\sup_{0 \leq t \leq T} \|\sqrt{\rho} u\|_{L^2}^2 + \int^T_0 \left[ \mu \|\nabla u\|_{L^2}^2 + (\mu + \lambda ) \|\divv u\|_{L^2}^2 \right] \mathrm{d} t \leq K \mathbb{F}_1 (0),
\\ \label{5.3}
\sup_{0 \leq t \leq T} \mathbb{F}_2 (t) + \int^T_0 \left(\|\sqrt{\rho} \dot{u}\|_{L^2}^2 + \Psi \kappa \|\nabla \theta \|_{L^2}^2 \right) \mathrm{d} t \leq K \mathbb{F}_2 (0),
\\ \label{5.4}
\sup_{0 \leq t \leq T}\big(t \mathbb{F}_2 (t)\big)+ \int^T_0 t \left(\|\sqrt{\rho} \dot{u}\|_{L^2}^2 + \Psi \kappa \|\nabla \theta \|_{L^2}^2 \right) \mathrm{d} t \leq K \Theta^2 \bar{\rho} \mathbb{F}_1 (0),
\\ \label{5.5}
\sup_{0 \leq t \leq T}\big(t \mathbb{F}_3 (t)\big)+ \int^T_0 t \left({\rm c_V} \|\sqrt{\rho} \dot{\theta}\|_{L^2}^2 + \mu \|\nabla \dot{u} \|_{L^2}^2 \right) \mathrm{d} t \leq K \mathbb{F}_2 (0),
\\ \label{5.6}
\sup_{0 \leq t \leq T}\big(t^2 \mathbb{F}_3 (t)\big)+ \int^T_0 t^2 \left({\rm c_V} \|\sqrt{\rho} \dot{\theta}\|_{L^2}^2 + \mu \|\nabla \dot{u} \|_{L^2}^2 \right) \mathrm{d} t \leq K \Theta^2 \bar{\rho} \mathbb{F}_1 (0),\\ \label{5.7}
\sup_{0 \leq t \leq T} \mathbb{F} (t) \leq \frac{1}{2} \mathbb{F}_0^{\frac{1}{2}},\ \ \Theta \int^T_0 \|\divv u\|_{L^{\infty}} \mathrm{d} t \leq \frac{1}{2} \mathbb{F}_1 (0)^{\frac{1}{32}},
\end{gather}
provided that $\mathbb{F}_0 \triangleq \mathbb{F}(0) \leq \varepsilon.$
\end{Proposition}

\begin{proof}
{\bf Step I. Proof of (\ref{5.2}).}
From \eqref{1.1}$_1$, (\ref{5.1}), and Gronwall's inequality, one sees that
\begin{equation}\label{5.8}
\sup_{0 \leq t \leq T} \|\rho\|_{L^p}\leq \|\rho_0\|_{L^p} \exp \left\{\int^T_0 \|\divv u\|_{L^{\infty}} {\rm d} t \right\} \leq K \bar{\rho}, \ \ \forall \ \frac{3}{2} \leq p \leq \infty.
\end{equation}
Moreover, it follows from \eqref{1.1}$_3$ and the maximum principle (see, e.g., \cite{HL2018}) that
$$
\inf_{\mathbb{R}^3 \times [0,T]} \theta (x,t) \geq 0.
$$
This combined with (\ref{1.1})$_2$ and (\ref{1.1})$_3$ gives
\begin{equation}\label{5.9}
\sup_{0 \leq t \leq T} \left(\|\sqrt \rho u \|_{L^2}^2 + {\rm c_V} \|\rho \theta \|_{L^1} \right) \leq K \mathbb{F}_1 (0).
\end{equation}
Furthermore, taking the inner product of (\ref{1.1})$_2$ with $u$ in $L^2$, we find that
\begin{equation}\label{5.10}
\frac{1}{2} \frac{\mathrm{d}}{\mathrm{d} t} \|\sqrt\rho u\|_{L^2}^2 + \mu \|\nabla u\|_{L^2}^2 + (\mu + \lambda) \|\divv u\|_{L^2}^2 = \int P \divv u \mathrm{d} x \leq {\rm c_V} \|\rho \theta\|_{L^1} \Theta \|\divv u\|_{L^{\infty}}.
\end{equation}
Integrating \eqref{5.10} with respect to $t$ and utilizing \eqref{5.9}, one can easily get \eqref{5.2}.

{\bf Step II. Proof of \eqref{5.3}.}
Taking the inner product of (\ref{1.1})$_3$ with $2 \theta $ in $L^2$, we obtain that
\begin{align}\label{5.11}
& {\rm c_V} \frac{\mathrm{d}}{\mathrm{d} t} \|\sqrt\rho \theta \|_{L^2}^2 + 2 \kappa \|\nabla \theta \|_{L^2}^2\notag \\
& \quad = - 2 \int P \divv u \theta \mathrm{d} x + 4 \mu \int |D(u)|^2 \theta \mathrm{d} x + 2 \lambda \int (\divv u)^2 \theta \mathrm{d} x \notag \\
& \quad \leq K {\rm R} \|\nabla u\|_{L^2} \|\rho \theta \|_{L^3} \|\theta \|_{L^6} + K \mu \|\nabla u\|_{L^2} \|\nabla u\|_{L^3} \|\theta \|_{L^6} + K (\mu + \lambda ) \|\divv u\|_{L^2} \|\divv u\|_{L^3} \|\theta \|_{L^6} \notag \\
& \quad \triangleq H_1 + H_2 + H_3.
\end{align}
Now we estimate the terms $H_1,H_2$, and $H_3$ term by term.
Taking advantage of \eqref{5.8}, Lemma \ref{lem3.4}, and \eqref{3.7} with $p = 6$, one has that
\begin{equation}\label{5.12}
\|\nabla u\|_{L^6}
\leq K \mu^{-1} \big( \bar{\rho}^{\frac{1}{2}} \|\sqrt{\rho} \dot{u}\|_{L^2} + {\rm R} \bar{\rho} \|\nabla \theta \|_{L^2} \big),
\end{equation}
which combined with Lemma \ref{lem3.4} and Young's inequality leads to
\begin{align*}
H_1
& \leq \frac{\kappa}{12} \|\nabla \theta \|_{L^2}^2 + K {\rm R}^2 \kappa^{-1} \|\nabla u\|_{L^2}^2 \|\rho \theta \|_{L^2} \|\rho \theta \|_{L^6} \leq \frac{\kappa}{6} \|\nabla \theta \|_{L^2}^2 + K {\rm R}^4 \kappa^{-3} \bar{\rho}^3 \|\sqrt{\rho} \theta \|_{L^2}^2 \|\nabla u\|_{L^2}^4\\
& \leq \frac{\kappa}{6} \|\nabla \theta \|_{L^2}^2 + K \mu^{-2} \Theta^{10} \bar{\rho}^3 \left({\rm c_V} \|\sqrt{\rho} \theta \|_{L^2}^2 \right) \|\sqrt{\mu}\nabla u\|_{L^2}^4,\\
H_2
& \leq \frac{\kappa}{12} \|\nabla \theta \|_{L^2}^2 + K \mu^2 \kappa^{-1} \|\nabla u\|_{L^2}^3 \|\nabla u\|_{L^6}\\
& \leq \frac{\kappa}{6} \|\nabla \theta \|_{L^2}^2 + K \delta \Psi^{-1} \|\sqrt{\rho} \dot{u}\|_{L^2}^2 + K (\delta) \mu^2 \kappa^{-2} \Psi \bar{\rho} \|\nabla u\|_{L^2}^6 + K \mu^2 \kappa^{-3} {\rm R}^2 \bar{\rho}^2 \|\nabla u\|_{L^2}^6\\
& \leq \frac{\kappa}{6} \|\nabla \theta \|_{L^2}^2 + K \delta \Psi^{-1} \|\sqrt{\rho} \dot{u}\|_{L^2}^2 + K (\delta) \mu^{-2} \Theta^{10} \bar{\rho}^3 \Psi^{-1} \|\sqrt{\mu}\nabla u\|_{L^2}^6,\\
H_3
& \leq \frac{\kappa}{12} \|\nabla \theta \|_{L^2}^2 + K \left(\mu + \lambda \right)^2 \kappa^{-1} \|\divv u\|_{L^2}^3 \|\divv u\|_{L^6}\\
& \leq \frac{\kappa}{12} \|\nabla \theta \|_{L^2}^2 + K \frac{\mu + |\lambda| }{\mu} \mu \kappa^{-1} \|\na u\|_{L^2}^3 \left( \bar{\rho}^{\frac{1}{2}} \|\sqrt{\rho} \dot{u}\|_{L^2} + {\rm R} \bar{\rho} \|\nabla \theta \|_{L^2} \right)\\
& \leq \frac{\kappa}{6} \|\nabla \theta \|_{L^2}^2 + K \delta \Psi^{-1} \|\sqrt{\rho} \dot{u}\|_{L^2}^2 + K (\delta) \mu^2 \Theta^2 \kappa^{-2} \Psi \bar{\rho} \|\na u\|_{L^2}^6 + K \mu^2 \Theta^2 \kappa^{-3} {\rm R}^2 \bar{\rho}^2 \|\na u\|_{L^2}^6 \\
& \leq \frac{\kappa}{6} \|\nabla \theta \|_{L^2}^2 + K \delta \Psi^{-1} \|\sqrt{\rho} \dot{u}\|_{L^2}^2 +  K (\delta) \mu^{-2} \Theta^{12} \bar{\rho}^3 \Psi^{-1} \|\sqrt{\mu}\na u\|_{L^2}^6.
\end{align*}
Substituting the above estimates into \eqref{5.11}, one arrives at
\begin{equation}\label{5.13}
\Psi {\rm c_V} \frac{\mathrm{d}}{\mathrm{d} t} \|\sqrt{\rho} \theta \|_{L^2}^2 + \Psi \kappa \|\nabla \theta \|_{L^2}^2 \leq K (\delta) \mu^{-2} \Theta^{12} \bar{\rho}^3 \mathbb{F}_2 (t) \|\sqrt{\mu}\na u\|_{L^2}^4 + K \delta \|\sqrt{\rho} \dot{u}\|_{L^2}^2.
\end{equation}

Next, multiplying (\ref{1.1})$_2$ by $u_t$ in $L^2$, one has
\begin{equation}\label{5.14}
\frac{1}{2} \frac{\mathrm{d}}{\mathrm{d} t} \int \left[ \mu |\nabla u|^2 + (\mu + \lambda )(\divv u)^2 \right] \mathrm{d} x + \int \rho |\dot{u}|^2 \mathrm{d} x = \int P \divv u_t \mathrm{d} x + \int \rho u \cdot \nabla u \cdot \dot{u} \mathrm{d} x.
\end{equation}
Owing to $\divv u = \frac{F + P }{2 \mu + \lambda }$, we find that
\begin{align}
\int P \divv u_t \mathrm{d} x
& = \frac{\mathrm{d}}{\mathrm{d} t} \int P \divv u \mathrm{d} x - \int P_t \divv u \mathrm{d} x\notag\\
& = \frac{\mathrm{d}}{\mathrm{d} t} \int P \divv u \mathrm{d} x - \frac{1}{2 (2 \mu + \lambda )} \frac{\mathrm{d}}{\mathrm{d} t} \int P^2 \mathrm{d} x - \frac{1}{2\mu + \lambda } \int P_t F \mathrm{d} x.\label{5.15}
\end{align}
Note that (\ref{1.1})$_3$ indicates
\begin{equation}\label{5.16}
P_t = (\gamma - 1) \left(\kappa \Delta \theta + 2 \mu |D(u)|^2 + \lambda(\divv u)^2 - P\divv u\right) - \divv(P u).
\end{equation}
Putting (\ref{5.16}) into (\ref{5.15}), we have
\begin{align}
\int P \divv u_t \mathrm{d} x
& = \frac{\mathrm{d}}{\mathrm{d} t} \int P \divv u \mathrm{d} x - \frac{1}{2 (2 \mu + \lambda )} \frac{\mathrm{d}}{\mathrm{d} t} \int P^2 \mathrm{d} x - \frac{1}{2\mu + \lambda } \int \left[ \kappa (\gamma - 1) \nabla \theta - P u \right] \cdot \nabla F \mathrm{d} x\notag\\
& \quad + \frac{\gamma - 1}{2 \mu + \lambda } \int \left(2 \mu |D(u)|^2 + \lambda(\divv u)^2 - P \divv u \right) F \mathrm{d} x.\label{5.17}
\end{align}
Inserting \eqref{5.17} into \eqref{5.14} yields
\begin{align}
& \frac{\mathrm{d}}{\mathrm{d} t} \int \left[ \mu |\nabla u|^2 + (\mu + \lambda ) (\divv u)^2 + \frac{P^2}{2\mu + \lambda } - 2 P \divv u \right] \mathrm{d} x + 2 \int \rho |\dot{u}|^2 \mathrm{d} x\notag\\
& \quad = \frac{2\gamma - 2}{2 \mu + \lambda } \int \left[2 \mu |D(u)|^2 + \lambda(\divv u)^2 - P \divv u \right] F \mathrm{d} x +2 \int \rho u \cdot \nabla u \cdot \dot{u} \mathrm{d} x\notag\\
& \qquad - \frac{2}{2\mu + \lambda } \int \left[ \kappa (\gamma - 1) \nabla \theta - P u \right] \cdot \nabla F \mathrm{d} x \triangleq I_1 + I_2 + I_3.\label{5.18}
\end{align}

It deduces from Lemma \ref{lem3.4}, \eqref{5.12}, H\"{o}lder's inequality, and Young's inequality that
\begin{align*}
& \frac{2\gamma - 2}{2 \mu + \lambda } \int \left[2 \mu |D(u)|^2 + \lambda(\divv u)^2 \right] F \mathrm{d} x \leq K \Theta^2 \|\nabla u\|_{L^2} \|\nabla u\|_{L^3} \|\nabla F\|_{L^2}\\
& \quad \leq \frac{1}{24} \|\sqrt \rho \dot{u}\|_{L^2}^2 + K \Theta^4 \bar{\rho} \|\nabla u\|_{L^2}^3 \|\nabla u\|_{L^6}\\
& \quad \leq \frac{1}{12} \|\sqrt \rho \dot{u}\|_{L^2}^2 + K \mu^{-1} \Theta^2 \bar{\rho} \|\nabla \theta \|_{L^2}^2 + K \left(\mu^{-2} \Theta^8 \bar{\rho}^3 + \mu^{-1} \Theta^{10} \bar{\rho}^3 \right) \|\nabla u\|_{L^2}^6\\
& \quad \leq \frac{1}{12} \|\sqrt \rho \dot{u}\|_{L^2}^2 + K \Psi \kappa \|\nabla \theta \|_{L^2}^2 + K \mu^{-4} \Theta^{10} \bar{\rho}^3 \mathbb{F}_2 (t) \|\sqrt{\mu}\nabla u\|_{L^2}^4,
\end{align*}
and
\begin{align*}
- 2 \frac{\gamma - 1}{2 \mu + \lambda } \int P\divv u F \mathrm{d} x
& \leq K \mu^{-1} \Theta^2 {\rm c_V} \|\nabla u\|_{L^2} \|\rho \theta \|_{L^3} \|\nabla F\|_{L^2}\\
& \leq \frac{1}{12} \|\sqrt \rho \dot{u}\|_{L^2}^2 + K \mu^{-2} \Theta^4 {\rm c_V}^2 \bar{\rho} \|\nabla u\|_{L^2}^2  \|\rho \theta \|_{L^2} \|\rho \theta \|_{L^6}\\
& \leq \frac{1}{12} \|\sqrt \rho \dot{u}\|_{L^2}^2 + K \mu^{-4} \Theta^6 \bar{\rho}^3 \mathbb{F}_2 (t) \|\sqrt{\mu}\nabla u\|_{L^2}^4 + K \Psi \kappa \|\nabla \theta \|_{L^2}^2.
\end{align*}
Thus, one gets that
$$
I_1 \leq \frac{1}{12} \|\sqrt \rho \dot{u}\|_{L^2}^2 + K \Psi \kappa \|\nabla \theta \|_{L^2}^2 + K \mu^{-4} \Theta^{10} \bar{\rho}^3 \mathbb{F}_2 (t) \|\mu^{\frac{1}{2}} \nabla u\|_{L^2}^4.
$$
For $I_2$ and $I_3$, it follows from Lemma \ref{lem3.4}, \eqref{5.12}, H\"{o}lder's inequality, and Young's inequality that
\begin{align*}
I_2
& \leq K \bar{\rho}^{\frac{1}{2}} \|\sqrt \rho \dot{u}\|_{L^2} \|\nabla u\|_{L^3} \|u\|_{L^6} \leq \frac{1}{6} \|\sqrt \rho \dot{u}\|_{L^2}^2 + K \Psi \kappa \|\nabla \theta \|_{L^2}^2 + K \mu^{-4} \Theta^2 \bar{\rho}^3 \mathbb{F}_2 (t) \|\sqrt{\mu}\nabla u\|_{L^2}^4,\\
I_3
& \leq K \mu^{-1} \|\nabla F\|_{L^2}  \left[ \kappa \left(\gamma - 1 \right)\|\nabla \theta \|_{L^2} + {\rm c_V} \left( \gamma - 1\right) \|\nabla u\|_{L^2} \|\rho \theta \|_{L^3} \right]\\
& \leq \frac{1}{6} \|\sqrt \rho \dot{u}\|_{L^2}^2 + K \mu^{-2} {\rm c_V}^2 \left( \gamma - 1\right)^2 \bar{\rho} \|\rho \theta \|_{L^2} \|\rho \theta \|_{L^6} \|\nabla u\|_{L^2}^2 + K \mu^{-2} \kappa^2 \left( \gamma - 1\right)^2 \bar{\rho} \|\nabla \theta \|_{L^2}^2\\
& \leq \frac{1}{6} \|\sqrt \rho \dot{u}\|_{L^2}^2 + K \Psi \kappa \|\nabla \theta \|_{L^2}^2 + K \mu^{-4} \Theta^2 \bar{\rho}^3 \mathbb{F}_2 (t) \|\sqrt{\mu}\nabla u\|_{L^2}^4.
\end{align*}
Substituting the above estimates on $I_i\ (i=1,2,3)$ into \eqref{5.18} and integrating the resultant over $(0,T)$, one deduces that, for a suitably large constant $K_1 \geq 2$,
\begin{align*}
& \sup_{0 \leq t \leq T} \left[\mu \|\nabla u\|_{L^2}^2 + \left(\mu + \lambda \right) \|\divv u\|_{L^2}^2 \right] + \int^T_0 \|\sqrt{\rho} \dot{u}\|_{L^2}^2 \mathrm{d} t\\
& \quad \leq K \mathbb{F}_2 (0) + K \bar{\rho}^{\frac{1}{2}} {\rm R} \sup_{0 \leq t \leq T}\big(\|\nabla u\|_{L^2} \|\sqrt{\rho} \theta \|_{L^2}\big)+ K \Psi \kappa \int^T_0 \|\nabla \theta \|_{L^2}^2 \mathrm{d} t \\
& \qquad + K \mu^{-4} \Theta^{10} \bar{\rho}^3 \int^T_0 \mathbb{F}_2 (t) \|\sqrt{\mu}\nabla u\|_{L^2}^4 \mathrm{d} t\\
& \quad \leq K \mathbb{F}_2 (0) + \frac{K_1}{2} \sup_{0 \leq t \leq T} \left( \Psi {\rm c_V} \|\sqrt{\rho} \theta \|_{L^2}^2 \right) + \frac{K_1}{2} \Psi \kappa \int^T_0 \|\nabla \theta \|_{L^2}^2 \mathrm{d} t\\
& \qquad + K \mu^{-4} \Theta^{10} \bar{\rho}^3 \int^T_0 \mathbb{F}_2 (t) \|\sqrt{\mu}\nabla u\|_{L^2}^4 \mathrm{d} t+ \frac{\mu}{2} \sup_{0 \leq t \leq T}\|\nabla u\|_{L^2}^2,
\end{align*}
which combined with \eqref{5.13} multiplied by $K_1$ and choosing $\delta \leq (2 K_2)^{-1}$ leads to
\begin{align*}
& \sup_{0 \leq t \leq T} \mathbb{F}_2 (t) + \int^T_0 \|\sqrt{\rho} \dot{u}\|_{L^2}^2 \mathrm{d} t + \Psi \kappa \int^T_0 \|\nabla \theta \|_{L^2}^2 \mathrm{d} t\\
& \quad \leq K \mathbb{F}_2 (0) +  K_2 \delta \int^T_0 \|\sqrt{\rho} \dot{u}\|_{L^2}^2 \mathrm{d} t + K \mu^{-2} \Theta^{12} \bar{\rho}^3 \int^T_0 \mathbb{F}_2 (t) \|\sqrt{\mu}\nabla u\|_{L^2}^4 \mathrm{d} t\\
& \quad \leq K \mathbb{F}_2 (0) + \frac{1}{2} \int^T_0 \|\sqrt{\rho} \dot{u}\|_{L^2}^2 \mathrm{d} t + K \mu^{-2} \Theta^{12} \bar{\rho}^3 \sup_{0 \leq t \leq T}\big(\mathbb{F}_1 (t) \mathbb{F}_2 (t)^2\big)\\
& \quad \leq K \mathbb{F}_2 (0) + \frac{1}{2} \int^T_0 \|\sqrt{\rho} \dot{u}\|_{L^2}^2 \mathrm{d} t + K_3 \mathbb{F}_0^{\frac{1}{2}} \sup_{0 \leq t \leq T} \mathbb{F}_2 (t)\\
& \quad \leq K \mathbb{F}_2 (0) + \frac{1}{2} \int^T_0 \|\sqrt{\rho} \dot{u}\|_{L^2}^2 \mathrm{d} t + \frac{1}{2} \sup_{0 \leq t \leq T} \mathbb{F}_2 (t),
\end{align*}
provided that
$$
\mathbb{F}_0 \leq \varepsilon_1 \triangleq \min \left\{1, (2 K_3)^{-2}\right\}.
$$
This gives the desired \eqref{5.3}.

{\bf Step III. Proof of \eqref{5.4}.}
Multiplying \eqref{5.18} by $t$ and integrating the resulting equality over $(0,T)$, we arrive at
\begin{align*}
& \sup_{0 \leq t \leq T} \left[\mu t \|\nabla u\|_{L^2}^2 + \left(\mu + \lambda \right) t\|\divv u\|_{L^2}^2 \right] + \int^T_0 t \|\sqrt{\rho} \dot{u}\|_{L^2}^2 \mathrm{d} t\\
& \quad \leq K \mathbb{F}_1 (0) + K \sup_{0 \leq t \leq T}\big(t \|\nabla u\|_{L^2} \|P\|_{L^2}\big)+ K \Psi \kappa \int^T_0 t \|\nabla \theta \|_{L^2}^2 \mathrm{d} t + K \mu^{-4} \Theta^{10} \bar{\rho}^3 \int^T_0 t \mathbb{F}_2 (t) \|\sqrt{\mu}\nabla u\|_{L^2}^4 \mathrm{d} t\\
& \quad \leq \frac{K_1}{2} \sup_{0 \leq t \leq T} \left(t \Psi {\rm c_V} \|\sqrt{\rho} \theta \|_{L^2}^2 \right) + \frac{K_1}{2} \Psi \kappa \int^T_0 t \|\nabla \theta \|_{L^2}^2 \mathrm{d} t + K \mu^{-4} \Theta^{10} \bar{\rho}^3 \int^T_0 t \mathbb{F}_2 (t) \|\sqrt{\mu}\nabla u\|_{L^2}^4 \mathrm{d} t\\
& \qquad + K \mathbb{F}_1 (0) + \frac{\mu}{2} \sup_{0 \leq t \leq T} \left(t \|\nabla u\|_{L^2}^2 \right).
\end{align*}
This together with \eqref{5.13} multiplied by $t K_1$ shows that
\begin{align*}
& \sup_{0 \leq t \leq T}\big(t \mathbb{F}_2 (t)\big)+ \int^T_0 t \|\sqrt{\rho} \dot{u}\|_{L^2}^2 \mathrm{d} t + \Psi \kappa \int^T_0 t \|\nabla \theta \|_{L^2}^2 \mathrm{d} t\\
& \quad \leq K \mathbb{F}_1 (0) +  \frac{1}{2} \int^T_0 t \|\sqrt{\rho} \dot{u}\|_{L^2}^2 \mathrm{d} t + K \mu^{-2} \Theta^{12} \bar{\rho}^3 \int^T_0 t \mathbb{F}_2 (t) \|\sqrt{\mu}\nabla u\|_{L^2}^4 \mathrm{d} t + K {\rm c_V} \Psi \int^T_0 \|\sqrt{\rho} \theta\|_{L^2}^2 {\rm d} t\\
& \quad \leq K \mathbb{F}_1 (0) + \frac{1}{2} \int^T_0 t \|\sqrt{\rho} \dot{u}\|_{L^2}^2 \mathrm{d} t + K \mu^{-2} \Theta^{12} \bar{\rho}^3 \sup_{0 \leq t \leq T}\big(t \mathbb{F}_1 (t) \mathbb{F}_2 (t)^2\big)+ K {\rm c_V} \Psi \bar{\rho} \int^T_0 \|\na \theta\|_{L^2}^2 {\rm d} t\\
& \quad \leq K \Theta^2 \bar{\rho} \mathbb{F}_1 (0) + \frac{1}{2} \int^T_0 t \|\sqrt{\rho} \dot{u}\|_{L^2}^2 \mathrm{d} t + K_3 \mathbb{F}_0^{\frac12} \sup_{0 \leq t \leq T}\big(t \mathbb{F}_2 (t)\big)\\
& \quad \leq K \Theta^2 \bar{\rho} \mathbb{F}_1 (0) + \frac{1}{2} \int^T_0 t \|\sqrt{\rho} \dot{u}\|_{L^2}^2 \mathrm{d} t + \frac{1}{2} \sup_{0 \leq t \leq T}\big(t \mathbb{F}_2 (t)\big),
\end{align*}
which implies \eqref{5.4}.

{\bf Step IV. Proof of \eqref{5.5} and \eqref{5.6}.}
 Multiplying \eqref{1.1}$_3$ by $\dot\theta$ in $L^2$ leads to
\begin{align}
& \frac{\kappa}{2} \frac{{\rm d}}{{\rm d} t} \int |\nabla\theta|^2 {\rm d} x + {\rm c_V} \int  \rho \dot{\theta}^2 {\rm d} x\notag\\
&\quad = -\frac{\kappa}{2} \int \left( 2 \partial_i \theta \partial_iu^j \partial_j \theta - \divv u |\nabla \theta|^2\right){\rm d} x + \int \left(\lambda (\divv u)^2 + 2 \mu |D(u)|^2 \right) \dot{\theta} {\rm d} x -{\rm R} \int \rho \theta \divv u \dot{\theta} {\rm d} x\notag\\
&\quad \triangleq J_1 + J_2 + J_3.\label{5.19}
\end{align}
We deduce from \eqref{1.1}$_3$ and \eqref{5.8} that
\begin{align}
\|\na^2 \theta\|_{L^2}^2
& \leq K \kappa^{-2}  \left({\rm c_V}^2 \|\rho \dot{\theta}\|_{L^2}^2 + \|P \divv u\|_{L^2}^2 + \mu^2 \|\na u\|_{L^4}^4 \right) \notag\\
& \leq K \kappa^{-2} \left( {\rm c_V}^2 \bar{\rho} \|\sqrt{\rho} \dot{\theta} \|_{L^2}^2 + \Theta^4 \bar{\rho}^2 \|\theta\|_{L^{\infty}}^2 \|\na u\|_{L^2}^2 + \mu^2 \|\na u\|_{L^4}^4 \right) \notag\\
& \leq \frac{1}{2} \|\na^2 \theta\|_{L^2}^2 + K \kappa^{-4} \Theta^8 \bar{\rho}^4 \|\na \theta\|_{L^2}^2 \|\na u\|_{L^2}^4 + K \kappa^{-2} {\rm c_V}^2 \bar{\rho} \|\sqrt{\rho} \dot{\theta}\|_{L^2}^2 + K \mu^2 \kappa^{-2} \|\na u\|_{L^4}^4.\label{5.20}
\end{align}
Then it follows from \eqref{5.20} that
\begin{align*}
J_1
& \leq K \kappa \|\na \theta\|_{L^2} \|\na u\|_{L^3} \|\na \theta\|_{L^6}\\
& \leq \frac{{\rm c_V}}{6} \|\sqrt{\rho} \dot{\theta}\|_{L^2}^2 + K \Theta^6 \bar{\rho}^2 \kappa \|\na \theta\|_{L^2}^2 \left(\|\na u\|_{L^2}^4 + \|\na u\|_{L^3}^2 \right) + K \mu^2 \kappa^{-1} \|\na u\|_{L^3}^2 \|\na u\|_{L^6}^2\\
& \leq \frac{{\rm c_V}}{6} \|\sqrt{\rho} \dot{\theta}\|_{L^2}^2 + K \Theta^6 \bar{\rho}^2 \kappa \|\na \theta\|_{L^2}^2 \left(\|\na u\|_{L^2}^4 + \|\na u\|_{L^3}^2 \right) + K \mu^{-1} \Theta^4 \bar{\rho}^2 \Psi^{-1} \|\sqrt{\rho} \dot{u}\|_{L^2}^2 \|\na u\|_{L^3}^2,\\
J_2
& \leq \frac{{\rm d}}{{\rm d} t} \int \left(\lambda (\divv u)^2 + 2 \mu |D(u)|^2 \right) \theta {\rm d} x + K \frac{\mu + |\lambda|}{\mu} \mu \|\na \theta\|_{L^2} \|\na u\|_{L^3} \left(\|\na \dot{u}\|_{L^2} + \|\na u\|_{L^4}^2 \right)\\
& \leq \frac{{\rm d}}{{\rm d} t} \int \left(\lambda ({\rm div} u)^2 + 2 \mu |D(u)|^2 \right) \theta {\rm d} x + \delta \mu \Psi^{-1} \|\na \dot{u}\|_{L^2}^2 + K \mu^{-1} \Theta^4 \bar{\rho}^2 \Psi^{-1} \|\sqrt{\rho} \dot{u}\|_{L^2}^2 \|\na u\|_{L^3}^2 \\
& \quad + K \Theta^6 \bar{\rho}^2 \kappa \|\na \theta\|_{L^2}^2 \|\na u\|_{L^3}^2,
\\
J_3 &\leq K \Theta^{\frac{3}{2}} {\rm c_V}^{\frac{1}{2}} \bar{\rho}^{\frac{1}{2}} \|\sqrt{\rho} \dot{\theta}\|_{L^2} \|\na u\|_{L^3} \|\theta\|_{L^6} \leq \frac{{\rm c_V}}{4} \|\sqrt{\rho} \dot{\theta}\|_{L^2}^2 + K \Theta^4 \bar{\rho} \kappa \|\na \theta\|_{L^2}^2 \|\na u\|_{L^3}^2.
\end{align*}
Plugging the above estimates on $J_k\ (k=1,2,3)$ into \eqref{5.19} yields that
\begin{align}
& \Psi \kappa \frac{{\rm d}}{{\rm d} t} \int |\nabla\theta|^2 {\rm d} x + \Psi {\rm c_V} \int \rho \dot{\theta}^2 {\rm d} x \notag\\
&\quad \leq 2 \Psi \frac{{\rm d}}{{\rm d} t} \int \left(\lambda (\divv u)^2 + 2 \mu |D(u)|^2 \right) \theta {\rm d} x + 2 \delta \mu \|\na \dot{u}\|_{L^2}^2 + K \Theta^6 \bar{\rho}^2 \mathbb{F}_3 (t) \left(\|\na u\|_{L^2}^4 + \|\na u\|_{L^3}^2 \right).\label{5.21}
\end{align}

Operating $\partial_t+ \divv(u\cdot)$ to both sides of the $j$-th equation of (\ref{1.1})$_2$, multiplying the resultant by $\dot u^j$ in $L^2$, and adding them together, we arrive at
\begin{align}
\frac{1}{2} \frac{{\rm d}}{{\rm d} t} \int \rho |\dot u|^2 {\rm d} x
& = \mu \int \dot u^j \big[\Delta u^j_t + \partial_k (u^k \Delta u^j) \big] {\rm d} x + (\mu+\lambda) \int \dot u^j \big[\partial_j \divv u_{t} + \partial_k(u^k \partial_j \divv u) \big] {\rm d} x \notag\\
& \quad - \int \dot u^j \big[ \big( \partial_j P \big)_t + \partial_k \big(u^k \partial_j P \big) \big] {\rm d} x \triangleq \sum_{i=1}^3 L_i.\label{5.22}
\end{align}
Integration by parts together with H{\"o}lder's inequality leads to
\begin{align*}
L_1
& = - \mu \int \big(|\nabla \dot u|^2 - \partial_k \dot u^j \partial_k u^l \partial_l u^j + \partial_l \dot u^j \partial_k u^k \partial_l u^j - \partial_k \dot u^j \partial_l u^k \partial_{l} u^j \big) {\rm d} x\\
& \leq - \frac{5\mu}{6} \|\nabla \dot u\|_{L^2}^2 + K \mu^{-1} \|\na u\|_{L^3}^2 \left(\bar{\rho} \|\sqrt{\rho} \dot{u}\|_{L^2}^2 + \Theta^4 \bar{\rho}^2 \|\na \theta\|_{L^2}^2 \right)\\
& \leq - \frac{5\mu}{6} \|\nabla \dot u\|_{L^2}^2 + K \Theta^3 \bar{\rho} \mathbb{F}_3 (t) \|\na u\|_{L^3}^2,\\
L_2
& \leq - \left(\mu + \lambda \right) \|\divv \dot{u}\|_{L^2}^2 + K \left(\mu + \lambda \right) \|\na \dot{u}\|_{L^2} \|\na u\|_{L^4}^2\\
& \leq - \left(\mu + \lambda \right) \|\divv \dot{u}\|_{L^2}^2 + \frac{\mu}{6} \|\nabla \dot u\|_{L^2}^2 + K \Theta^5 \bar{\rho} \mathbb{F}_3 (t) \|\na u\|_{L^3}^2, \\
L_3
& = - \int \dot u^j \big[ \partial_j \big({\rm R} \rho \dot \theta \big) - \partial_k \big(P \partial_j u^k \big) \big] {\rm d} x\\
&\leq K {\rm c_V}(\gamma - 1) \bar{\rho}^{\frac{1}{2}} \|\nabla \dot u\|_{L^2} \|\sqrt\rho \dot \theta\|_{L^2} + K {\rm R} \bar{\rho} \|\nabla\dot u\|_{L^2} \|\nabla u\|_{L^3} \|\nabla\theta\|_{L^2} \\
&\leq \frac{\mu}{8} \|\nabla\dot u\|_{L^2}^2 + K \mu^{-1} {\rm c_V}^2(\gamma - 1)^2 \bar{\rho} \| \sqrt\rho \dot \theta\|_{L^2}^2 + K \mu^{-1} {\rm R}^2 \bar{\rho}^2 \|\nabla \theta\|_{L^2}^2 \|\na u\|_{L^3}^2\\
&\leq \frac{\mu}{8} \|\nabla\dot u\|_{L^2}^2 + K \Psi {\rm c_V} \| \sqrt\rho \dot \theta\|_{L^2}^2 + K \Theta^2 \bar{\rho} \mathbb{F}_3 (t) \|\na u\|_{L^3}^2.
\end{align*}
Plugging the above estimates into \eqref{5.22} and multiplying the resultant by $t$, we obtain after integrating it with respect to $t$ that for a suitably large positive constant $K_4 \geq 2$,
\begin{align*}
& \sup_{0 \leq t \leq T}\big(t \|\sqrt{\rho} \dot{u}\|_{L^2}^2\big)+ \mu \int^T_0 t \|\na \dot{u}\|_{L^2}^2 {\rm d} t + \left( \mu + \lambda \right) \int^T_0 t \|\divv \dot{u}\|_{L^2}^2 {\rm d} t\\
& \quad \leq K \int^T_0 \|\sqrt{\rho} \dot{u}\|_{L^2}^2 {\rm d} t + \frac{K_4}{2} \Psi {\rm c_V} \int^T_0 t \|\sqrt{\rho} \dot{\theta}\|_{L^2}^2 {\rm d} t + K \Theta^5 \bar{\rho} \int^T_0 t \mathbb{F}_3 (t) \|\na u\|_{L^3}^2 {\rm d} t.
\end{align*}
This combined with (\ref{5.21}) multiplied by $K_4 t$ indicates that
\begin{align*}
& \sup_{0 \leq t \leq T}\big(t \mathbb{F}_3(t)\big)+ \mu \int^T_0 t \|\na \dot{u}\|_{L^2}^2 {\rm d} t + \left( \mu + \lambda \right) \int^T_0 t \|\divv \dot{u}\|_{L^2}^2 {\rm d} t + \frac{K_4}{2} \Psi {\rm c_V} \int^T_0 t \|\sqrt{\rho} \dot{\theta}\|_{L^2}^2 {\rm d} t \\
& \quad \leq K \int^T_0 \left(\|\sqrt{\rho} \dot{u}\|_{L^2}^2 + \Psi \kappa \|\na \theta\|_{L^2}^2 \right) {\rm d} t + K \Theta^6 \bar{\rho}^2 \int^T_0 t \mathbb{F}_3 (t) \left(\|\na u\|_{L^3}^2 + \|\na u\|_{L^2}^4 \right) {\rm d} t\\
& \qquad + K \Psi \Theta \mu \sup_{0 \leq t \leq T}\big(t \|\na u\|_{L^2} \|\na u\|_{L^3} \|\theta\|_{L^6}\big)+ K \Psi \Theta \mu \int^T_0 \|\na u\|_{L^2} \|\na u\|_{L^3} \|\theta\|_{L^6} {\rm d} t\\
& \quad \leq K \mathbb{F}_2 (0) + \frac{\Psi \kappa}{4} \sup_{0 \leq t \leq T} \big(t\|\na \theta\|_{L^2}^2\big)+ K \mu^{-2} \Theta^6 \bar{\rho}^2 \sup_{0 \leq t \leq T}\big(t \mathbb{F}_3 (t)\big)\sup_{0 \leq t \leq T} \|\sqrt{\mu}\na u\|_{L^2}^2 \int^T_0 \|\sqrt{\mu}\na u\|_{L^2}^2 {\rm d} t\\
& \qquad + K \mu^{-\frac{1}{2}} \Theta^6 \bar{\rho}^2 \sup_{0 \leq t \leq T}\big(t \mathbb{F}_3 (t)\big)\left(\int^T_0 \|\sqrt{\mu}\na u\|_{L^2}^2 {\rm d} t \right)^{\frac{1}{2}} \left[\int^T_0 \mu^{-2} \left(\bar{\rho} \|\sqrt{\rho} \dot{u}\|_{L^2}^2 + \Theta^4 \bar{\rho}^2 \|\na \theta\|_{L^2}^2 \right){\rm d} t \right]^{\frac{1}{2}}\\
& \qquad + K \int^T_0 \left( \Psi \kappa \|\na \theta\|_{L^2}^2 + \kappa^{-1} \mu^2 \Psi \Theta^2 \|\na u\|_{L^2}^3 \|\na u\|_{L^6} \right) {\rm d} t + K \mu^2 \kappa^{-1} \Psi \Theta^2 \sup_{0 \leq t \leq T} t \|\na u\|_{L^2}^3 \|\na u\|_{L^6}\\
& \quad \leq  K \mathbb{F}_0^{\frac{1}{4}} \Big(1 + \mathbb{F}_0^{\frac{1}{4}} \Big) \sup_{0 \leq t \leq T}\big(t \mathbb{F}_3 (t)\big)+ K \int^T_0 \left( \Psi \kappa \|\na \theta\|_{L^2}^2 + \|\sqrt{\rho} \dot{u}\|_{L^2}^2 + \mu^{-2} \Theta^{14} \bar{\rho}^3 \|\sqrt{\mu}\na u\|_{L^2}^6 \right) {\rm d} t\\
& \qquad + K \mathbb{F}_2 (0) + \frac{3}{8} \sup_{0 \leq t \leq T}\big(t \mathbb{F}_3 (t)\big) + K \mu^{-2} \Theta^{14} \bar{\rho}^3 \sup_{0 \leq t \leq T} \big(t \|\sqrt{\mu} \na u\|_{L^2}^6 \big) \\
& \quad \leq  K_5 \mathbb{F}_0^{\frac{1}{4}} \sup_{0 \leq t \leq T}\big(t \mathbb{F}_3 (t)\big)+ K \mathbb{F}_2 (0) + \frac{3}{8} \sup_{0 \leq t \leq T}\big(t \mathbb{F}_3 (t)\big)\leq K \mathbb{F}_2 (0) + \frac{1}{2} \sup_{0 \leq t \leq T}\big(t \mathbb{F}_3 (t)\big),
\end{align*}
provided that
$$
\mathbb{F}_0 \leq \varepsilon_2 \triangleq \min \left\{\varepsilon_1, (8 K_5)^{-4} \right\}.
$$
This leads to \eqref{5.5}.

Similarly, one has
\begin{align*}
& \sup_{0 \leq t \leq T}\big(t^2 \mathbb{F}_3(t)\big)+ \mu \int^T_0 t^2 \|\na \dot{u}\|_{L^2}^2 {\rm d} t + \left( \mu + \lambda \right) \int^T_0 t^2 \|\divv \dot{u}\|_{L^2}^2 {\rm d} t + \frac{K_4}{2} \Psi {\rm c_V} \int^T_0 t^2 \|\sqrt{\rho} \dot{\theta}\|_{L^2}^2 {\rm d} t \\
& \quad \leq K \int^T_0 t \left(\|\sqrt{\rho} \dot{u}\|_{L^2}^2 + \Psi \kappa \|\na \theta\|_{L^2}^2 \right) {\rm d} t + K \Theta^6 \bar{\rho}^2 \int^T_0 t^2 \mathbb{F}_3 (t) \left(\|\na u\|_{L^3}^2 + \|\na u\|_{L^2}^4 \right) {\rm d} t\\
& \qquad + K \Psi \Theta \mu \sup_{0 \leq t \leq T}\big(t^2 \|\na u\|_{L^2} \|\na u\|_{L^3} \|\theta\|_{L^6}\big)+ K \Psi \Theta \mu \int^T_0 t \|\na u\|_{L^2} \|\na u\|_{L^3} \|\theta\|_{L^6} {\rm d} t\\
& \quad \leq  K_5 \mathbb{F}_0^{\frac{1}{4}} \sup_{0 \leq t \leq T}\big(t^2 \mathbb{F}_3 (t)\big)+ K \int^T_0 t \big( \Psi \kappa \|\na \theta\|_{L^2}^2 + \|\sqrt{\rho} \dot{u}\|_{L^2}^2 + \mu^{-2} \Theta^{14} \bar{\rho}^3 \| \sqrt{\mu} \na u\|_{L^2}^6 \big) {\rm d} t\\
& \qquad + K \Theta^2 \bar{\rho} \mathbb{F}_1 (0) + \frac{3}{8} \sup_{0 \leq t \leq T} \big(t^2 \mathbb{F}_3 (t)\big)+ K \mu^{-2} \Theta^{14} \bar{\rho}^3\sup_{0 \leq t \leq T}\big(t^2 \|\sqrt{\mu}\na u\|_{L^2}^6\big) \\
& \quad \leq K \Theta^2 \bar{\rho} \mathbb{F}_1 (0) + \frac12\sup_{0 \leq t \leq T}\big(t^2 \mathbb{F}_3 (t)\big),
\end{align*}
which implies the desired \eqref{5.6}.

{\bf Step V. Proof of (\ref{5.7}).}
It follows from (\ref{5.3}) and (\ref{5.9}) that
$$
\mathbb{F} (t) \leq K_6 \mathbb{F}_0 \leq \frac{1}{2} \mathbb{F}_0^{\frac{1}{2}},
$$
provided that
$$
\mathbb{F}_0 \leq \varepsilon_3 \triangleq \min \left\{\varepsilon_2, (2 K_6)^{-2} \right\}.
$$

It follows from Lemma \ref{lem3.4}, \eqref{5.12}, and \eqref{5.20} that
\begin{align*}
\Theta \|{\rm div} u\|_{L^{\infty}}
& \leq K \mu^{-1} \Theta \|F\|_{L^6}^{\frac{1}{2}} \|\na F\|_{L^6}^{\frac{1}{2}} + K \mu^{-1} \Theta^3 \bar{\rho} \|\theta\|_{L^6}^{\frac{1}{2}} \|\na \theta\|_{L^6}^{\frac{1}{2}}\\
& \leq K \mu^{-1} \Theta \bar{\rho}^{\frac{3}{4}} \|\sqrt{\rho} \dot{u}\|_{L^2}^{\frac{1}{2}} \|\na \dot{u}\|_{L^2}^{\frac{1}{2}} +  K \mu^{-1} \Theta^3 \bar{\rho} \|\na \theta\|_{L^2}^{\frac{1}{2}} \|\na^2 \theta\|_{L^2}^{\frac{1}{2}}\\
& \leq K \mu^{-1} \Theta \bar{\rho}^{\frac{3}{4}} \|\sqrt{\rho} \dot{u}\|_{L^2}^{\frac{1}{2}} \|\na \dot{u}\|_{L^2}^{\frac{1}{2}} + K \mu^{-1} \kappa^{-1} \Theta^5 \bar{\rho}^2 \|\na \theta\|_{L^2} \|\na u\|_{L^2}\\
& \quad + K \mu^{-1} \kappa^{-\frac{1}{2}} \Theta^3 {\rm c_V}^{\frac{1}{2}} \bar{\rho}^{\frac{5}{4}} \|\na \theta\|_{L^2}^{\frac{1}{2}} \|\sqrt{\rho} \dot{\theta}\|_{L^2}^{\frac{1}{2}} + K \mu^{-\frac{1}{2}} \kappa^{-\frac{1}{2}} \Theta^3 \bar{\rho} \|\na \theta\|_{L^2}^{\frac{1}{2}} \|\na u\|_{L^4}\\
& \leq K \mu^{-\frac{5}{4}} \Theta \bar{\rho}^{\frac{3}{4}} \|\sqrt{\rho} \dot{u}\|_{L^2}^{\frac{1}{2}} \|\mu^{\frac{1}{2}}\na \dot{u}\|_{L^2}^{\frac{1}{2}} + K \mu^{-1} \Theta^5 \bar{\rho}^{\frac{3}{2}} \|\Psi^{\frac{1}{2}} \kappa^{\frac{1}{2}} \na \theta\|_{L^2} \|\mu^{\frac{1}{2}}\na u\|_{L^2}\\
& \quad + K \mu^{-1} \kappa^{-\frac{1}{2}} \Theta^3 {\rm c_V}^{\frac{1}{2}} \bar{\rho}^{\frac{5}{4}} \|\na \theta\|_{L^2}^{\frac{1}{2}} \|\sqrt{\rho} \dot{\theta}\|_{L^2}^{\frac{1}{2}} + K \mu^{-\frac{1}{2}} \kappa^{-\frac{1}{2}} \Theta^3 \bar{\rho} \|\na \theta\|_{L^2}^{\frac{1}{2}} \|\na u\|_{L^2}^{\frac{1}{4}} \|\na u\|_{L^6}^{\frac{3}{4}} \\
& \leq K \mu^{-\frac{1}{2}} \Theta^{\frac{5}{2}} \bar{\rho}^{\frac{3}{4}} \|(\sqrt{\rho} \dot{u}, \Psi^{\frac{1}{2}} \kappa^{\frac{1}{2}} \na \theta)\|_{L^2}^{\frac{1}{2}} \|(\mu^{\frac{1}{2}}\na \dot{u}, \Psi^{\frac{1}{2}} {\rm c_V}^{\frac{1}{2}} \sqrt{\rho} \dot{\theta})\|_{L^2}^{\frac{1}{2}}\\
& \quad + K \mu^{-1} \Theta^5 \bar{\rho}^{\frac{3}{2}} \|\Psi^{\frac{1}{2}} \kappa^{\frac{1}{2}} \na \theta\|_{L^2} \|\mu^{\frac{1}{2}}\na u\|_{L^2} + K \mu^{-\frac{3}{4}} \Theta^{\frac{15}{4}} \bar{\rho}^{\frac{9}{8}} \|(\sqrt{\rho} \dot{u}, \Psi^{\frac{1}{2}} \kappa^{\frac{1}{2}} \na \theta)\|_{L^2}^{\frac{5}{4}} \|\mu^{\frac{1}{2}}\na u\|_{L^2}^{\frac{1}{4}}.
\end{align*}
In order to prove the second estimate in \eqref{5.7}, we denote by $\Lambda \triangleq \min\{\mu^{-1} \mathbb{F}_1 (0)^2, T\}$. Then it follows from (\ref{5.2})--(\ref{5.6}) that
\begin{align*}
& K \mu^{-\frac{1}{2}} \Theta^{\frac{5}{2}} \bar{\rho}^{\frac{3}{4}} \int^T_0 \|(\sqrt{\rho} \dot{u}, \Psi^{\frac{1}{2}} \kappa^{\frac{1}{2}} \na \theta)\|_{L^2}^{\frac{1}{2}} \|(\mu^{\frac{1}{2}}\na \dot{u}, \Psi^{\frac{1}{2}} {\rm c_V}^{\frac{1}{2}} \sqrt{\rho} \dot{\theta})\|_{L^2}^{\frac{1}{2}} {\rm d} t \\
& \quad \leq K \mu^{-\frac{1}{2}} \Theta^{\frac{5}{2}} \bar{\rho}^{\frac{3}{4}} \int^{\Lambda}_0 \|(\sqrt{\rho} \dot{u}, \Psi^{\frac{1}{2}} \kappa^{\frac{1}{2}} \na \theta)\|_{L^2}^{\frac{1}{2}} t^{\frac{1}{4}} \|(\mu^{\frac{1}{2}}\na \dot{u}, \Psi^{\frac{1}{2}} {\rm c_V}^{\frac{1}{2}} \sqrt{\rho} \dot{\theta})\|_{L^2}^{\frac{1}{2}} t^{-\frac{1}{4}} {\rm d} t\\
& \qquad + K \mu^{-\frac{1}{2}} \Theta^{\frac{5}{2}} \bar{\rho}^{\frac{3}{4}} \int_{\Lambda}^T t^{\frac{1}{4}} \|(\sqrt{\rho} \dot{u}, \Psi^{\frac{1}{2}} \kappa^{\frac{1}{2}} \na \theta)\|_{L^2}^{\frac{1}{2}} t^{\frac{3}{8}} \|(\mu^{\frac{1}{2}}\na \dot{u}, \Psi^{\frac{1}{2}} {\rm c_V}^{\frac{1}{2}} \sqrt{\rho} \dot{\theta})\|_{L^2}^{\frac{1}{2}} t^{- \frac{5}{8}} {\rm d} t\\
& \quad \leq K \mu^{-\frac{1}{2}} \Theta^{\frac{5}{2}} \bar{\rho}^{\frac{3}{4}} \left(\int^{\Lambda}_0 \|(\sqrt{\rho} \dot{u}, \Psi^{\frac{1}{2}} \kappa^{\frac{1}{2}} \na \theta)\|_{L^2}^2 {\rm d} t \right)^{\frac{1}{4}} \left(\int^{\Lambda}_0  t \|(\mu^{\frac{1}{2}}\na \dot{u}, \Psi^{\frac{1}{2}} {\rm c_V}^{\frac{1}{2}} \sqrt{\rho} \dot{\theta})\|_{L^2}^2 {\rm d} t \right)^{\frac{1}{4}} \left(\int^{\Lambda}_0 t^{-\frac{1}{2}} {\rm d} t \right)^{\frac{1}{2}}\\
& \qquad + K \mu^{-\frac{1}{2}} \Theta^{\frac{5}{2}} \bar{\rho}^{\frac{3}{4}} \left(\int_{\Lambda}^T t \|(\sqrt{\rho} \dot{u}, \Psi^{\frac{1}{2}} \kappa^{\frac{1}{2}} \na \theta)\|_{L^2}^2 {\rm d} t \right)^{\frac{1}{4}} \left(\int_{\Lambda}^T t \|(\mu^{\frac{1}{2}}\na \dot{u}, \Psi^{\frac{1}{2}} {\rm c_V}^{\frac{1}{2}} \sqrt{\rho} \dot{\theta})\|_{L^2}^2 {\rm d} t \right)^{\frac{1}{8}}\\
& \qquad \cdot \left( \int_{\Lambda}^T t^2 \|(\mu^{\frac{1}{2}}\na \dot{u}, \Psi^{\frac{1}{2}} {\rm c_V}^{\frac{1}{2}} \sqrt{\rho} \dot{\theta})\|_{L^2}^2 {\rm d} t\right)^{\frac{1}{8}} \left(\int_{\Lambda}^T t^{- \frac{5}{4}} {\rm d} t \right)^{\frac{1}{2}}\\
& \quad \leq K \mu^{-\frac{3}{4}} \Theta^{\frac{5}{2}} \bar{\rho}^{\frac{3}{4}} \mathbb{F}_1 (0)^{\frac{1}{2}} \mathbb{F}_2 (0)^{\frac{1}{2}} + K \mu^{-\frac{3}{8}} \Theta^{\frac{13}{4}} \bar{\rho}^{\frac{9}{8}} \mathbb{F}_1 (0)^{\frac{1}{8}} \mathbb{F}_2 (0)^{\frac{1}{8}}\\
& \quad \leq K \mu^{-\frac{3}{4}} \Theta^{\frac{5}{2}} \bar{\rho}^{\frac{3}{4}} \sup_{0 \leq t \leq T}\Big(\mathbb{F}_1 (t)^{\frac{1}{2}} \mathbb{F}_2 (t)^{\frac{1}{2}}\Big)+ K \mu^{-\frac{3}{8}} \Theta^{\frac{13}{4}} \bar{\rho}^{\frac{9}{8}} \sup_{0 \leq t \leq T}\Big(\mathbb{F}_1 (t)^{\frac{1}{8}} \mathbb{F}_2 (t)^{\frac{1}{8}}\Big)\leq K \mathbb{F}_0^{\frac{1}{16}},\\
& K \mu^{-1} \Theta^5 \bar{\rho}^{\frac{3}{2}} \int^T_0 \|\Psi^{\frac{1}{2}} \kappa^{\frac{1}{2}} \na \theta\|_{L^2} \|\mu^{\frac{1}{2}}\na u\|_{L^2} {\rm d} t \leq K \mu^{-1} \Theta^5 \bar{\rho}^{\frac{3}{2}} \left(\int^T_0 \|\Psi^{\frac{1}{2}} \kappa^{\frac{1}{2}} \na \theta\|_{L^2}^2 {\rm d} t \right)^{\frac{1}{2}} \left(\int^T_0 \|\mu^{\frac{1}{2}}\na u\|_{L^2}^2 {\rm d} t \right)^{\frac{1}{2}}\\
&\quad \leq K \mu^{-1} \Theta^5 \bar{\rho}^{\frac{3}{2}} \sup_{0 \leq t \leq T} \Big(\mathbb{F}_1 (t)^{\frac{1}{2}} \mathbb{F}_2 (t)^{\frac{1}{2}}\Big)\leq K \mathbb{F}_0^{\frac{1}{4}},
\end{align*}
and
\begin{align*}
& K \mu^{-\frac{3}{4}} \Theta^{\frac{15}{4}} \bar{\rho}^{\frac{9}{8}} \int^T_0 \|(\sqrt{\rho} \dot{u}, \Psi^{\frac{1}{2}} \kappa^{\frac{1}{2}} \na \theta)\|_{L^2}^{\frac{5}{4}} \|\mu^{\frac{1}{2}}\na u\|_{L^2}^{\frac{1}{4}} {\rm d} t\\
& \quad \leq K \mu^{-\frac{3}{4}} \Theta^{\frac{15}{4}} \bar{\rho}^{\frac{9}{8}} \left(\int_{\Lambda}^T t \|(\sqrt{\rho} \dot{u}, \Psi^{\frac{1}{2}} \kappa^{\frac{1}{2}} \na \theta)\|_{L^2}^2 {\rm d} t \right)^{\frac{1}{2}}  \left(\int_{\Lambda}^T \|(\sqrt{\rho} \dot{u}, \Psi^{\frac{1}{2}} \kappa^{\frac{1}{2}} \na \theta)\|_{L^2}^2 {\rm d} t \right)^{\frac{1}{8}} \left(\int_{\Lambda}^T \|\mu^{\frac{1}{2}}\na u\|_{L^2}^2 {\rm d} t \right)^{\frac{1}{8}} \\
& \qquad \cdot \left(\int_{\Lambda}^T t^{-2} {\rm d} t \right)^{\frac{1}{4}} + K \mu^{-\frac{3}{4}} \Theta^{\frac{15}{4}} \bar{\rho}^{\frac{9}{8}} \left(\int^{\Lambda}_0 \|(\sqrt{\rho} \dot{u}, \Psi^{\frac{1}{2}} \kappa^{\frac{1}{2}} \na \theta)\|_{L^2}^2 {\rm d} t \right)^{\frac{5}{8}} \left(\int^{\Lambda}_0 \|\mu^{\frac{1}{2}}\na u\|_{L^2}^2 {\rm d} t \right)^{\frac{1}{8}} \left(\int^{\Lambda}_0 {\rm d} t \right)^{\frac{1}{4}}\\
& \quad \leq K \mu^{-1} \Theta^{\frac{15}{4}} \bar{\rho}^{\frac{9}{8}} \sup_{0 \leq t \leq T}\Big(\mathbb{F}_1 (t)^{\frac{5}{8}} \mathbb{F}_2 (t)^{\frac{5}{8}}\Big)+ K \mu^{-\frac{1}{2}} \Theta^{\frac{19}{4}} \bar{\rho}^{\frac{13}{8}} \sup_{0 \leq t \leq T} \Big(\mathbb{F}_1 (t)^{\frac{1}{8}} \mathbb{F}_2 (t)^{\frac{1}{8}}\Big)\leq K \mathbb{F}_0^{\frac{1}{16}}.
\end{align*}
Therefore, one can easily get
\begin{align*}
\Theta \int^T_0 \|\divv u\|_{L^{\infty}} {\rm d} t \leq K \mathbb{F}_0^{\frac{1}{16}} \bigg(1 + \mathbb{F}_0^{\frac{3}{16}} \bigg) \leq K_7 \mathbb{F}_0^{\frac{1}{16}} \leq \frac{1}{2} \mathbb{F}_0^{\frac{1}{32}},
\end{align*}
provided that
$$
\mathbb{F}_0 \leq \varepsilon \triangleq \min \left\{\varepsilon_3, (2 K_7)^{-32} \right\}.
$$
The proof of Proposition \ref{pro5.1} is complete.
\end{proof}

With the global {\it a priori} estimates in Proposition \ref{pro5.1} at hand, we now show Theorem \ref{thm2.2}.

{\bf Proof of Theorem \ref{thm2.2}}. The global existence of solutions can be obtained from Proposition \ref{pro5.1} and blow-up criterion \eqref{3.3}.
Indeed, taking $s=4$ and $r=6$ in \eqref{3.2} and using \eqref{5.2}, \eqref{5.3}, and \eqref{5.7}, there exists a positive constant $C$ independent of $T$ such that
$$
\int_{0}^T\|\divv u\|_{L^\infty} {\rm d}t+\int_{0}^T\|u\|_{L^6}^4{\rm d}t \le
\int_{0}^T \|\divv u\|_{L^\infty} {\rm d} t + C \sup_{0\leq t\leq T}\|\nabla u\|_{L^2}^{2}
\int_{0}^T \|\nabla u\|_{L^2}^{2}{\rm d}t \le C,
$$
which combined with \eqref{3.3} implies that $T=\infty$. Furthermore, the proof of decay rate \eqref{2.15} is similar to that in \cite{XZ2025}, and we omit the details for simplicity. This completes the proof of Theorem \ref{thm2.2}.
\hfill$\square$

\section*{Authors' Contributions}
H.X. and X.Z. wrote the main manuscript text. All authors reviewed the manuscript.

\subsection*{Conflict of Interest}
The authors declare that they have no conflict of interest.

\subsection*{Data Availability}
No data was used for the research described in the article.

\section*{Acknowledgments}
The authors would like to express their gratitude to the reviewers for careful reading and helpful suggestions which led to an improvement of the original manuscript. At the same time, the authors have great admiration for the reviewers' professionalism.


\begin{thebibliography}{99}

\bibitem{CD10}
F. Charve and R. Danchin, A global existence result for the compressible Navier--Stokes equations in the critical $L^p$ framework, \textit{Arch. Ration. Mech. Anal.,} \textbf{198} (2010), 233--271.

\bibitem{CCZ10}
Q. Chen, C. Miao, and Z. Zhang, Global well-posedness for compressible Navier--Stokes equations with highly oscillating initial velocity, {\it Comm. Pure Appl. Math.,} \textbf{63} (2010), 1173--1224.

\bibitem{Da00}
R. Danchin, Global existence in critical spaces for compressible Navier--Stokes equations, {\it Invent. Math.,} \textbf{141} (2000), 579--614.

\bibitem{DM17}
R. Danchin and P. B. Mucha, Compressible Navier--Stokes system: large solutions and incompressible limit, {\it Adv. Math.}, {\bf320} (2017), 904--925.

\bibitem{DM23}
R. Danchin and P. B. Mucha, Compressible Navier--Stokes equations with ripped density, {\it Comm. Pure Appl. Math.,} \textbf{76} (2023), 3437--3492.

\bibitem{Fe2004} E. Feireisl,
Dynamics of viscous compressible fluids,
Oxford University Press, Oxford, 2004.

\bibitem{EF01}
E. Feireisl, A. Novotn\'{y}, and H. Petzeltov\'a, On the existence of globally defined weak solutions to the Navier--Stokes equations, \textit{J. Math. Fluid Mech.,} \textbf{3} (2001), 358--392.

\bibitem{Hoff95}
D. Hoff, Global solutions of the Navier--Stokes equations for multidimensional compressible flow with discontinuous initial data, \textit{J. Differential Equations,} \textbf{120} (1995), 215--254.

\bibitem{Hoff95*}
D. Hoff, Strong convergence to global solutions for multidimensional flows of compressible, viscous fluids with polytropic equations of state and discontinuous initial data, \textit{Arch. Ration. Mech. Anal.,} \textbf{132} (1995), 1--14.

\bibitem{Ho1997} D. Hoff,
Discontinuous solutions of the Navier--Stokes equations for multidimensional flows of heat-conducting fluids,
{\it Arch. Ration. Mech. Anal.}, {\bf 139} (1997), 303--354.

\bibitem{HP24} G. Hong, X. Hou, H. Peng, and C. Zhu,
Global existence for a class of large solution to compressible Navier--Stokes equations with vacuum,
{\it Math. Ann.}, {\bf 388} (2024), 2163--2194.

\bibitem{Hu2020} X. Huang,
On local strong and classical solutions to the three-dimensional barotropic compressible Navier--Stokes equations with vacuum,
{\it Sci. China Math.}, {\bf 64} (2021), 1771--1788.

\bibitem{HL2018} X. Huang and J. Li,
Global classical and weak solutions to the three-dimensional full compressible Navier--Stokes system with vacuum and large oscillations,
{\it Arch. Ration. Mech. Anal.}, {\bf 227} (2018), 995--1059.

\bibitem{HLW13} X. Huang, J. Li, and Y. Wang,
Serrin-type blowup criterion for full compressible Navier--Stokes system,
{\it Arch. Ration. Mech. Anal.}, {\bf 207} (2013), 303--316.

\bibitem{HLX11} X. Huang, J. Li, and Z. Xin,
Serrin-type criterion for the three-dimensional viscous compressible flows,
{\it SIAM J. Math. Anal.}, {\bf 43} (2011), 1872--1886.

\bibitem{HLX2012} X. Huang, J. Li, and Z. Xin,
Global well-posedness of classical solutions with large oscillations and vacuum to the three-dimensional isentropic compressible Navier--Stokes equations,
{\it Comm. Pure Appl. Math.}, {\bf 65} (2012), 549--585.

\bibitem{JZ01}
S. Jiang and P. Zhang, On spherically symmetric solutions of the compressible isentropic Navier--Stokes equations, \textit{Comm. Math. Phys.,} \textbf{215} (2001), 559--581.

\bibitem{JZ03}
S. Jiang and P. Zhang, Axisymmetric solutions of the 3D Navier--Stokes equations for compressible isentropic fluids, {\it J. Math. Pures Appl.}, {\bf 82} (2003), 949--973.

\bibitem{LXZ2022} S. Lai, H. Xu, and J. Zhang,
Well-posedness and exponential decay for the Navier--Stokes equations of viscous compressible heat-conductive fluids with vacuum,
{\it Math. Models Meth. Appl. Sci.}, {\bf 32} (2022), 1725--1784.

\bibitem{LX2019-1} Z. Lei and Z. Xin,
On scaling invariance and type-I singularities for the compressible Navier--Stokes equations, {\it Sci. China Math.}, {\bf 62} (2019), 2271--2286.

\bibitem{book17}
G. Leoni, A first course in Sobolev spaces, 2nd ed., American Mathematical Society, Providence, RI, 2017.

\bibitem{LX2019} J. Li and Z. Xin,
Global well-posedness and large time asymptotic behavior of classical solutions to the compressible Navier--Stokes equations with vacuum,
{\it Ann. PDE}, {\bf 5} (2019), Paper No. 7.

\bibitem{Li2020} J. Li,
Global small solutions of heat conductive compressible Navier--Stokes equations with vacuum: smallness on scaling invariant quantity,
{\it Arch. Ration. Mech. Anal.}, {\bf 237} (2020), 899--919.

\bibitem{LZ2023} J. Li and Y. Zheng,
Local existence and uniqueness of heat conductive compressible Navier--Stokes equations in the presence of vacuum without initial compatibility conditions,
{\it J. Math. Fluid Mech.}, {\bf 25} (2023), Paper No. 14.

\bibitem{L21} Z. Liang,
Global strong solutions of Navier--Stokes equations for heat-conducting compressible fluids with vacuum at infinity,
{\it J. Math. Fluid Mech.}, {\bf 23} (2021), Paper No. 17.

\bibitem{Li1998} P.-L. Lions,
Mathematical topics in fluid mechanics. Vol. 2. Compressible models,
Oxford University Press, New York, 1998.

\bibitem{MN1980} A. Matsumura and T. Nishida,
The initial value problem for the equations of motion of viscous and heat-conductive gases, {\it J. Math. Kyoto Univ.}, {\bf 20} (1980), 67--104.

\bibitem{M1}
F. Merle, P. Rapha\"el, I. Rodnianski, and J. Szeftel, On the implosion of a compressible fluid I: smooth self-similar inviscid profiles,
{\it Ann. of Math.}, {\bf 196} (2022), 567--778.

\bibitem{M2}
F. Merle, P. Rapha\"el, I. Rodnianski, and J. Szeftel, On the implosion of a compressible fluid II: singularity formation, {\it Ann. of Math.}, {\bf 196} (2022), 779--889.

\bibitem{W2025} H. Wen,
Global wellposedness of compressible Navier--Stokes equations with vacuum and smallness on scaling invariant quantity in $\mathbb{R}^3$,
{\it Adv. Math.}, {\bf 482} (2025), Paper No. 110628.

\bibitem{WZ2017} H. Wen and C. Zhu,
Global solutions to the three-dimensional full compressible Navier--Stokes equations with vacuum at infinity in some classes of large data,
{\it SIAM J. Math. Anal.}, {\bf 49} (2017), 162--221.

\bibitem{XZ2025} H. Xu and J. Zhang,
Global stability of large solutions to the three-dimensional full compressible Navier--Stokes equations with vacuum,
{\it J. Math. Phys.}, {\bf 66} (2025), Paper No. 081517.

\bibitem{Z2000} A. A. Zlotnik,
Uniform estimates and stabilization of symmetric solutions of a system of quasilinear equations,
{\it Differ. Equ.}, {\bf 36} (2000), 701--716.

\end{thebibliography}
\end{document}